\documentclass[10pt,twocolumn,twoside]{IEEEtran}
\usepackage{caption}
\usepackage{bbm}
\usepackage{amsmath}
\usepackage{amsfonts}
\usepackage{mathrsfs}
\usepackage{amssymb}
\usepackage{multirow}
\usepackage{latexsym}
\usepackage{psfrag}
\usepackage{graphicx}
\usepackage{cite}
\usepackage{algorithm}
\usepackage{algpseudocode}
\usepackage{color}
\usepackage{times,balance}
\usepackage{algorithm,multirow,xcolor}
\usepackage{algorithmicx}
\usepackage{algpseudocode}
\usepackage{longtable}
\usepackage{booktabs}
\usepackage{colortbl}
\usepackage{bm}
\usepackage{booktabs}
\usepackage{amssymb}
\usepackage{pifont}
\usepackage{wasysym}
\usepackage{threeparttable}
\usepackage{comment}
\usepackage{subcaption}
\usepackage{enumitem}
\usepackage{soul}
\usepackage{cuted, caption}
\usepackage{graphicx}      
\usepackage{subcaption}   
\usepackage{float}
\DeclareMathOperator*{\argmin}{arg\,min}

\definecolor{mygray}{gray}{.8}

\newtheorem{theorem}{Theorem}
\newtheorem{definition}{Definition}

\newtheorem{assumption}{Assumption}

\newtheorem{lemma}{Lemma}
\newtheorem{remark}{Remark}

\allowdisplaybreaks[4]

\begin{document}

\title{Heterogeneous Stochastic Momentum ADMM for Distributed Nonconvex Composite Optimization}
\author{Yangming Zhang, Yongyang Xiong, Jinming Xu, Keyou You,  
and Yang Shi,  \IEEEmembership{Fellow, IEEE}
\thanks{This work was supported in part by the National Natural Science Foundation
of China (62203254).
\emph{(Corresponding author: Yongyang Xiong.)}}
\thanks{Y. Zhang and Y. Xiong are with the School of Intelligent Systems Engineering, Sun Yat-Sen University, Shenzhen 518107, P.R China. E-mail:  \texttt{xiongyy25@mail.sysu.edu.cn}.}
\thanks{
J. Xu is with the State Key Laboratory of Industrial Control Technology and the College of Control Science and Engineering, Zhejiang University, Hangzhou 310027, China. E-mail: \texttt{jimmyxu@zju.edu.cn}.}
\thanks{K. You is with the Department of Automation, Beijing National Research Center for Information Science and Technology, Tsinghua University, Beijing 100084, China. E-mail: \texttt{youky@tsinghua.edu.cn}.}
\thanks{Y. Shi is with the Department of Mechanical Engineering, University of Victoria, Victoria, BC V8W 2Y2, Canada. E-mail: \texttt{yshi@uvic.ca}.}
}
\maketitle

\begin{abstract}
This paper investigates the distributed stochastic nonconvex and nonsmooth composite optimization problem. Existing stochastic typically rely on uniform step size strictly bounded by global network parameters, such as the maximum node degree or spectral radius. This dependency creates a severe performance bottleneck, particularly in heterogeneous network topologies where the step size must be conservatively reduced to ensure stability. To overcome this limitation, we propose a novel Heterogeneous Stochastic Momentum Alternating Direction Method of Multipliers (HSM-ADMM). By integrating a recursive momentum estimator (STORM), HSM-ADMM achieves the optimal oracle complexity of $\mathcal{O}(\epsilon^{-1.5})$ to reach an $\epsilon$-stationary point, utilizing a strictly single-loop structure and an $\mathcal{O}(1)$ mini-batch size. The core innovation lies in a node-specific adaptive step-size strategy, which scales the proximal term according to local degree information. We theoretically demonstrate this design completely decouples the algorithmic stability from global network properties, enabling robust and accelerated convergence across arbitrary connected topologies without requiring any global structural knowledge. Furthermore, HSM-ADMM requires transmitting only a single primal variable per iteration, significantly reducing communication bandwidth compared to state-of-the-art gradient tracking algorithms. Extensive numerical experiments on distributed nonconvex learning tasks validate the superior efficiency of the proposed HSM-ADMM algorithm.
\end{abstract}


\section{Introduction}

\IEEEPARstart{D}{istributed} optimization has emerged as a fundamental paradigm for large-scale machine learning and signal processing, motivated by the need to process massive datasets distributed across multi-agent networks \cite{Zhang_Discrete-Time, XiongTAC2023, ChenTII2024} and increasing concerns about data privacy \cite{XiongTCNS2020, WuTAC2025}. Unlike centralized algorithms that require data aggregation at a central server, distributed optimization enables a network of $n$ agents to collaboratively minimize a global objective using only local computation and peer-to-peer communication. This paradigm not only mitigates the risk of single‑point failures and privacy leakage, but also offers superior scalability for massive-scale applications \cite{YANG2019278}.

In this paper, we focus on the following distributed composite optimization problem over an undirected and connected network of $n$ agents:
\begin{equation}\label{Original Problem}
\min_{x \in \mathbb{R}^p}\sum_{i=1}^n{\left( f_i\left( x \right)  +h_i\left( x \right) \right)},
\end{equation}
where $f_{i}:\mathbb{R}^p\to\mathbb{R}$ denotes a closed and smooth local loss function (possibly nonconvex), and $h_{i}:\mathbb{R}^p\to\mathbb{R}\cup \left \{ +\infty \right \}$ represents a convex but nonsmooth regularizer (e.g., the $\ell_1$-norm for sparsity). Specifically, we consider the stochastic setting, i.e., $f_i(x) \triangleq \mathbb {E}_{\xi_i\sim \mathcal {D}_i}[f_i(x,\xi_{i})]$, where $\xi_i$ is a random data sample drawn from an unknown local distribution $\mathcal {D}_i$, and $f_i(x,\xi_{i})$ is the corresponding sample loss. In this setting, each agent $i$ only has access to the stochastic gradients of its local function $f_i$ and handles the nonsmooth regularizer $h_i$ via its proximal mapping, a standard premise in distributed and privacy-sensitive scenarios. This formulation encompasses a wide range of applications, including distributed machine learning \cite{Yan2023CompressedDP}, sensor network estimation \cite{ChengTSP2021, XU2024111863}, and cooperative control of vehicle swarms \cite{NiuTSMC2025}.

Existing distributed optimization algorithms broadly fall into two categories: primal-domain methods and primal-dual methods. In the primal domain, early works primarily rely on distributed (sub)gradient descent (D-SGD) \cite{nedic2009distributed} and its variants \cite{sundhar2010distributed, yuan2016convergence}. These algorithms face a fundamental step size trade-off: a constant step size enables fast convergence but only to a neighborhood of the optimal solution due to steady-state error \cite{nedic2009distributed}, \cite{yuan2016convergence}, whereas a diminishing step size guarantees exact convergence at slow sublinear rate \cite{sundhar2010distributed}. In stochastic setting, these algorithms are further limited by the inherent variance of stochastic gradients. Although variance reduction techniques such as SVRG \cite{SVRG2013} and SPIDER \cite{SPIDER2018} mitigate this issue, they impose prohibitive computational burdens due to periodic full-gradient evaluations or nested double-loop structures. Recently, momentum-based recursive variance reduction techniques (e.g., STORM \cite{STORM2019}) have achieved optimal convergence rates with single-loop updates and $\mathcal{O}(1)$ mini-batch sizes. However, extending these benefits to distributed nonconvex and nonsmooth composite optimization remains largely underexplored.

\begin{table*}[t]
\caption{Comparisons of different algorithms}
\label{tab:Comparisons of different algorithms}
\center
\begin{threeparttable}  
\renewcommand{\arraystretch}{1.6}
\begin{tabular}[l]{ccccc}
\toprule
Algorithm    & Objective Function                & Batch Size                                & Sample Complexity(per agent)         & Communication Complexity \\ 
\midrule
DSGT\cite{DSGT_XinRan}         &$F(\mathbf {x})$                   &$\mathcal{O}\left(1\right)$            &$\mathcal{O}\left(\max\left\{\frac{1}{n\varepsilon^2},\frac{\rho n}{(1-\rho)^3\varepsilon}\right\}\right)$                              &$\mathcal{O}\left(\max\left\{\frac{1}{n\varepsilon^2},\frac{\rho n}{(1-\rho)^3\varepsilon}\right\}\right)$                          \\
SPPDM\cite{Wang2021TSP_SPPDM}        &$F(\mathbf {x}) + H(\mathbf {x})$  &$\Omega\left(\frac{n}{\epsilon}\right)$          & $\mathcal{O}\left(\frac{1}{\left(1-\rho\right)^b\epsilon ^2}\right)$                          &$\mathcal{O}\left(\frac{1}{\epsilon }\right)$                          \\
ProxGT-SA \cite{xin2021stochastic}   &$F(\mathbf {x}) + H(\mathbf {x})$       &$\mathcal{O}\left(\frac{1}{\epsilon}\right)$            &$\mathcal{O}\left(\frac{1}{n\epsilon ^2}\right)$                             &$\mathcal{O}\left(\frac{\log(n)}{\epsilon} \right)$                          \\
ProxGT-SR-O \cite{xin2021stochastic} &$F(\mathbf {x}) + H(\mathbf {x})$                    &$\mathcal{O}\left(\frac{1}{\epsilon}\right)$            &$\mathcal{O}\left(\frac{1}{n\epsilon ^{1.5}}\right)$                              &$\mathcal{O}\left(\frac{\log(n)}{\epsilon} \right)$                          \\
DEEPSTORMv2 \cite{AAAIDEEPSTORM} &$F(\mathbf {x}) + H(\mathbf {x})$                    &$\mathcal{O}\left(1\right)$            &$\mathcal{O}\left(\frac{1}{\epsilon^{1.5}|\log(\epsilon)|^{1.5}} \right)$                              &$\mathcal{O}\left(\frac{1}{\epsilon^{1.5}|\log(\epsilon)|^{1.5}} \right)$                          \\
Prox-DASA \cite{TesiXiao2023Prox-DASA}   &$F(\mathbf {x}) + H(\mathbf {x})$                    &$\mathcal{O}\left(1\right)$            &$\mathcal{O}\left(\frac{1}{n\epsilon ^2}\right)$                              &$\mathcal{O}\left(\frac{1}{n\epsilon ^2}\right)$                          \\
Prox-DASA-GT \cite{TesiXiao2023Prox-DASA} &$F(\mathbf {x}) + H(\mathbf {x})$                    &$\mathcal{O}\left(1\right)$            &$\mathcal{O}\left(\frac{1}{n\epsilon ^2}\right)$                              &$\mathcal{O}\left(\frac{1}{n\epsilon ^2}\right)$                          \\ 
\midrule
This work     &$F(\mathbf {x}) + H(\mathbf {x})$                    &$\mathcal{O}\left(1\right)$            &$\tilde{\mathcal{O}}\left(\frac{1}{\epsilon ^{1.5}}\right)$                              &$\tilde{\mathcal{O}}\left(\frac{1}{\epsilon ^{1.5}}\right)$                          \\ 
\bottomrule
\end{tabular}
\vspace{2pt}  
{\footnotesize \item[*]$\rho$ quantifies the connectivity efficiency of the communication graph and $b$ is a constant.}
\end{threeparttable}
\end{table*}

To gracefully handle nonsmooth regularizers and network consensus constraints, dual-domain methods—particularly the alternating direction method of multipliers (ADMM) and its variants—are widely adopted \cite{boyd2011distributed}. However, the direct application of ADMM to large-scale distributed problems is computationally expensive, as updating the primal variables requires solving multiple implicit subproblems per iteration. To mitigate this computational burden, various approximate distributed ADMM variants have been proposed. For example, the distributed linearized ADMM (DLM) \cite{ling2015dlm} linearizes the objective function in the augmented Lagrangian function. Similarly, the penalty ADMM (PAD) \cite{zhang2021penalty} obtains an inexact solution by incorporating a quadratic penalty term into the objective function and using an approximation of the augmented Lagrangian function in the iteration process. Moreover, the distributed proximal ADMM (DP-ADMM) \cite{zhou2024distributed} solves the multiblock nonsmooth composite optimization problem and extends it to asynchronous scenario. Nevertheless,  \cite{ling2015dlm, zhang2021penalty, zhou2024distributed} all focus on convex settings and lack explicit convergence guarantees for nonconvex stochastic scenarios. While the recent PPG-ADMM \cite{zhou2025perturbed} relaxes stringent smoothness assumptions for nonconvex composite optimization, it does not provide an explicit convergence rate for the stochastic distributed case. Furthermore, the centralized SMADMM \cite{deng2025stochasticmomentumadmmnonconvex} achieves the optimal rate $\tilde{\mathcal{O}}(\epsilon^{-1.5})$ for the nonconvex stochastic composite optimization problem, yet extending it to a distributed consensus setting is nontrivial.

Among recent advances in distributed nonconvex optimization, the SPPDM algorithm \cite{Wang2021TSP_SPPDM} incorporates Nesterov momentum to accelerate convergence. ProxGT-SA and ProxGT-SR-O \cite{xin2021stochastic} incorporate stochastic gradient tracking and multi-consensus update in proximal gradient methods. However, a critical limitation of both SPPDM and ProxGT-SA is that their batch sizes are inextricably coupled with the target accuracy $\epsilon$. For high-precision solutions, the batch size becomes prohibitively large, leading to excessive computational demands and degraded generalization performance. DEEPSTORM \cite{AAAIDEEPSTORM} circumvents this by introducing a recursive momentum estimator. A fundamental limitation of DEEPSTORM and most existing distributed algorithms is that they employ a uniform step-size across all agents, which must be strictly bounded by global network properties (e.g., the spectral gap or the maximum node degree). This ``one-size-fits-all'' strategy creates a severe performance bottleneck, particularly in heterogeneous networks containing both highly connected ``hub'' nodes and sparsely connected ``leaf'' nodes. To ensure stability, the uniform step size must be conservatively reduced to accommodate the most restrictive node (typically, the node with the maximum degree), thereby forcing capable agents to update unnecessarily slowly and degrading overall convergence speed. Additionally, algorithms like \cite{xin2021stochastic, AAAIDEEPSTORM, TesiXiao2023Prox-DASA} require multiple variable exchanges per iteration, escalating the communication overhead in bandwidth-limited environments.

To simultaneously address the bottlenecks of topology-dependent step-sizes, large batch sizes, and heavy communication overhead, we propose a novel Heterogeneous Stochastic Momentum ADMM (HSM-ADMM) for distributed nonconvex composite optimization. By integrating ADMM with recursive momentum estimation, our algorithm handles both nonsmooth regularization and stochastic variance, achieving the optimal oracle complexity of $\tilde{\mathcal{O}}(\epsilon^{-1.5})$ with $\mathcal{O}(1)$ batch sizes. Crucially, we introduce a node-specific adaptive step-size mechanism scaled by each agent’s local degree. This design fundamentally decouples local updates from global topological constraints, effectively immunizing the algorithm against the straggler effect in heterogeneous networks. Table \ref{tab:Comparisons of different algorithms} summarizes the comparison between our algorithm and related state-of-the-art algorithms.

We summarize our main contributions as follows:
\begin{enumerate}
\item[(1)]\textbf{Single-loop Algorithm with Optimal Complexity:} We propose HSM-ADMM, a fully distributed single-loop algorithm for nonconvex and nonsmooth composite optimization. By leveraging a recursive momentum estimator, HSM-ADMM achieves the optimal oracle complexity of $\mathcal{O}(\epsilon^{-1.5})$ to reach an $\epsilon$-stationary point. This matches the theoretical lower bound for first-order stochastic nonconvex optimization, while requiring only $\mathcal{O}(1)$ mini-batch size per iteration without periodic full-gradient evaluations.

\item[(2)]\textbf{Topology-Independent Heterogeneous Step-sizes:} We design a novel heterogeneous step size strategy where each agent's step size is determined solely by its local degree $d_i$, independent of the global spectral radius or maximum degree. Theoretical analysis proves that this design eliminates the ``straggler effect'' caused by network bottlenecks, enabling faster convergence in heterogeneous and sparse topologies compared to uniform step-size algorithms.

\item[(3)]\textbf{Communication Efficiency:} Unlike state-of-the-art gradient tracking-based methods (e.g., ProxGT-SA, ProxGT-SR-O \cite{xin2021stochastic}, and DEEPSTORM \cite{AAAIDEEPSTORM}) that require transmitting two variables (model parameter and gradient tracker) per iteration, HSM-ADMM maintains the minimal communication cost of transmitting only one variable per iteration. This structural advantage effectively halves the communication bandwidth consumption per iteration, making our algorithm particularly suitable for communication-constrained environments.

\item[(4)]\textbf{Extensive Empirical Validation:} Experiments on various nonconvex learning tasks and network topologies demonstrate that our algorithm consistently outperforms state-of-the-art baselines in terms of both convergence speed and communication efficiency.
\end{enumerate}

The rest of this paper is organized as follows.
Section \ref{sec-II} introduces the preliminaries and problem formulation. Section \ref{sec-III} details the proposed HSM-ADMM algorithm. Section \ref{sec-IV} presents the convergence analysis.
Section \ref{sec-V} provides numerical experiments and, Section \ref{sec-VI} concludes the paper.

{\bf Notations.} $\|\cdot\|$ denotes the $\ell_2$-norm for vectors and the Frobenius norm for matrices. $\|\cdot\|_2$ denotes the spectral norm for matrices. For a positive definite matrix $Q$, $\| x \|_Q=\sqrt{x^{\mathrm{T} }Qx} $
is the norm of the vector $x$ with respect to $Q$. $I$ and $\mathbf{0}$ denote an identity matrix and a zero matrix, respectively. The operator $\otimes$ denotes the Kronecker product. Each agent $i$ holds a local variable $x_i \in \mathbb{R}^p$, whose value in the $k$-th iteration is denoted by $x_i^k$. We denote ${\mathbf {x}}=[x_1^{\mathrm{T} },\ldots,x_n^{\mathrm{T} }]^{\mathrm{T} } \in \mathbb {R}^{n p}$, $F({\bf x})=\sum_{i=1}^n{  f_i\left( x \right) }$, and $H({\bf x})=\sum_{i=1}^n{  h_i\left( x \right) }$. The gradient of $F$ at ${\bf x}$ is denoted by
\begin{equation*}
\nabla F({\bf x})= [(\nabla f_1(x_1))^{\mathrm{T} },\ldots,(\nabla f_n(x_n))^{\mathrm{T} }]^{\mathrm{T} },
\end{equation*}
where $\nabla f_i(x_i)$ is the gradient of $f_i$ at $x_i$. We define the following proximal operator of  $h_i$
\begin{align*} 
{\rm prox}_{h_i}^c(x) \triangleq \argmin_u \left\lbrace h_i(u) + \frac{1}{2c} \Vert u-x \Vert ^2 \right\rbrace,
\end{align*}
where $c>0$ is a scalar parameter.
\section{Preliminaries and Problem Formulation}\label{sec-II}
\subsection{Graph Theory}
In this paper, we consider an undirected and connected graph $\mathcal{G}({\mathcal{V}},\mathcal{E})$ (i.e., there exists a communication path between any pair of nodes), where $\mathcal{V}=\{1,\dots,n\}$ is the set of nodes and $\mathcal{E}$ is the set of edges with $m = |\mathcal{E}|$. Let $\mathcal{N}_{i}\triangleq\{j \in \mathcal{V}:(j,i)\in\mathcal{E}\}\cup \{i\}$ denote the neighbors of node $i$ that include itself, and let $d_i = |\mathcal{N}_i| - 1$ be its degree. The node-edge incidence matrix $M \in \mathbb{R}^{m \times n}$ of the graph $\mathcal{G}$ is defined by assigning an arbitrary direction $i \to j$ to each $k$-th edge $e_k = (i, j)$; specifically, $M_{ki} = 1$, $M_{kj} = -1$, and all other entries in the $k$-th row are zero. The graph Laplacian matrix is given by $\mathbf{L}_G = M^{\mathrm{T}}M$.

\subsection{Problem Formulation}
We consider a networked system where $n$ agents collaboratively solve the distributed stochastic composite optimization problem \eqref{Original Problem}. Unlike static empirical risk minimization, we assume data samples arrive sequentially or are continuously drawn from an underlying distribution. Specifically, each local objective $f_i$ is defined as the expectation over a random variable $\xi_i$:
\begin{equation}\label{Problem_fi}
f_i(x) \triangleq \mathbb{E}_{\xi_i \sim \mathcal{D}_i}[f_i(x, \xi_i)],
\end{equation}
where $\xi_i$ follows an unknown probability distribution $\mathcal{D}_i$, and $f_i(x, \xi_i)$ is the instantaneous cost function for a sample $\xi_i$. This formulation effectively captures online scenarios where samples are generated in real-time, requiring agents to rely on stochastic approximations of $\nabla f_i(x)$ for optimization.

To solve this problem in a fully distributed manner via the ADMM framework, we introduce a local copy $x_i \in \mathbb{R}^p$ for each agent $i \in \mathcal{V}$, and employ auxiliary variables $y_i \in \mathbb{R}^p$ to decouple the nonsmooth regularization term $h_i(y_i)$ from the consensus constraints. The original problem \eqref{Original Problem} is thus equivalently reformulated as:
\begin{align}\label{Distributed_Problem}
\min_{\{x_i\}, \{y_i\}} \ & \sum_{i=1}^n \mathbb{E}_{\xi_i \sim \mathcal{D}_i}[f_i(x_i, \xi_i)] + \sum_{i=1}^n h_i(y_i) \notag \\
\text{s.t.} \ & x_i = x_j, \quad \forall (i,j) \in \mathcal{E} \notag \\
& x_i = y_i, \quad \forall i \in \mathcal{V}.
\end{align}
To facilitate algorithm design and theoretical analysis, we rewrite \eqref{Distributed_Problem} into a compact matrix form:
\begin{align}\label{compact_problem}
    \min_{\mathbf{x}, \mathbf{y}} \ & F(\mathbf{x}) + H(\mathbf{y}) \notag \\
    \text{s.t.} \ & A\mathbf{x} + B\mathbf{y} = \mathbf{0},
\end{align}
where $\mathbf{x} = [x_1^{\mathrm{T}}, \dots, x_n^{\mathrm{T}}]^{\mathrm{T}} \in \mathbb{R}^{np}$ and $\mathbf{y} = [y_1^{\mathrm{T}}, \dots, y_n^{\mathrm{T}}]^{\mathrm{T}} \in \mathbb{R}^{np}$. The global functions are defined as $F(\mathbf{x}) \triangleq \sum_{i=1}^n f_i(x_i)$ and $H(\mathbf{y}) \triangleq \sum_{i=1}^n h_i(y_i)$. The constraint matrices $A$ and $B$ are constructed as:
\begin{align}\label{eq:matrices_AB_def}
A &\triangleq \begin{bmatrix} M \otimes I_p \\ I_{np} \end{bmatrix} \in \mathbb{R}^{(m+n)p \times np}, \\
B &\triangleq \begin{bmatrix} \mathbf{0}_{mp \times np} \\ -I_{np} \end{bmatrix} \in \mathbb{R}^{(m+n)p \times np}.
\end{align}

The specific structure of $A$ imparts desirable spectral properties that are instrumental for our convergence analysis. Specifically, consider the matrix $A^{\mathrm{T}} A$:
\begin{equation}
A^{\mathrm{T}} A = (M^{\mathrm{T}} M \otimes I_p) + I_{np} = \mathbf{L}_G \otimes I_p + I_{np}.
\end{equation}
Let $0 = \lambda_1(\mathbf{L}_G) < \lambda_2(\mathbf{L}_G) \le \dots \le \lambda_n(\mathbf{L}_G)$ denote the eigenvalues of the Laplacian for the connected graph. The eigenvalues of $A^{\mathrm{T}} A$ are given by $\lambda_i(A^{\mathrm{T}} A) = \lambda_i(\mathbf{L}_G) + 1$. 
Since the Laplacian of an undirected graph is positive semi-definite with $\lambda_{\min}(\mathbf{L}_G) = 0$, we have the following critical identity:
\begin{equation}\label{eq:sigma_A_identity}
    \sigma_{\min}^2(A) = \lambda_{\min}(A^{\mathrm{T}} A) = \lambda_{\min}(\mathbf{L}_G) + 1 = 1.
\end{equation}
This implies that $A$ has full column rank and its smallest singular value is exactly 1, independent of the network topology.
\subsection{Assumptions}
To establish the theoretical guarantees of our algorithm, we adopt the following standard assumptions regarding the communication graph, the objective functions, and the stochastic oracle\cite{xin2021stochastic, AAAIDEEPSTORM}.
\begin{assumption}[Network Topology]\label{Ass: undirected and connected}
The communication graph $\mathcal{G}$ is undirected and connected.
\end{assumption}

\begin{assumption}[Smoothness]\label{Ass: mean-squared smoothness}
For all $i \in \mathcal{V}$, the local function $f_i(x)$ is mean-squared smooth with constant $L > 0$. That is, for any $x, y \in \mathbb{R}^p$:
\begin{equation*}
\mathbb{E}_{\xi_i}[\|\nabla f_i(x, \xi_i) - \nabla f_i(y, \xi_i)\|^2] \le L^2 \|x - y\|^2.
\end{equation*}
Note that this implies the standard $L$-smoothness of the expected function $f_i(x) \triangleq \mathbb{E}[f_i(x, \xi_i)]$.
\end{assumption}

\begin{assumption}[Lower Boundedness]\label{Ass: f,g lower bounded}
The global objective function is bounded from below, i.e., $ \inf_x \sum_{i=1}^n (f_i(x) + h_i(x)) > -\infty$.
\end{assumption}

\begin{assumption}[Stochastic Oracle]\label{Ass: unbiasedness and variance boundedness}
For each agent $i$, the stochastic gradient $\nabla f_i(x, \xi_i)$ is an unbiased estimator of the true gradient $\nabla f_i(x)$ with bounded variance:
\begin{align*}
&\mathbb{E}_{\xi_i}[\nabla f_i(x, \xi_i)] = \nabla f_i(x),\\
&\mathbb{E}_{\xi_i}[\|\nabla f_i(x, \xi_i) - \nabla f_i(x)\|^2] \le \sigma^2,
\end{align*}
where $\xi_i \sim \mathcal{D}_i$ represents the local data samples.
\end{assumption}

A critical observation regarding the aforementioned assumptions is the deliberate omission of the bounded data heterogeneity condition. Unlike a wide range of existing methods that require the spatial variance of local gradients to be strictly bounded (e.g., assuming $\frac{1}{n} \sum_{i=1}^n \|\nabla f_i(x) - \nabla F(x)\|^2 \le \zeta^2$)\cite{lian2017can}, our convergence analysis is completely independent of such topological data divergence $\zeta^2$. Consequently, HSM-ADMM is theoretically guaranteed to resist arbitrary degrees of data heterogeneity, making it highly practical for strictly Non-IID (Non-Independent and Identically Distributed‌) distributed learning problems.

\section{Algorithm Development}\label{sec-III}
In this section, we develop the proposed Heterogeneous Stochastic Momentum ADMM (HSM-ADMM) distributed algorithm. We first introduce the augmented Lagrangian function associated with the compact formulation \eqref{compact_problem}. Subsequently, we detail the construction of the stochastic momentum estimator and the adaptive primal-dual update rules. Finally, we present the fully distributed implementation of the proposed HSM-ADMM algorithm.
\subsection{Augmented Lagrangian Function}
To handle the constraints in \eqref{Distributed_Problem}, we introduce the Lagrange multipliers $\alpha_{ij}\in \mathbb{R}^p$ associated with the consensus constraints $x_i = x_j$ and Lagrange multipliers $\beta_i\in \mathbb{R}^p$ associated with the variable splitting constraints $x_i = y_i$. Grouping the multipliers $\alpha_{ij}$ in the vector $\boldsymbol{\alpha}=[\alpha_{ij}]_{(i,j)\in \mathcal {E}}\in \mathbb{R}^{mp}$ and the multipliers $\beta_i$ in the vector $\boldsymbol{\beta}=[\beta_{ij}]_{(i,j)\in \mathcal {E}}\in \mathbb{R}^{np}$, We concatenate $\boldsymbol{\alpha}$ and $\boldsymbol{\beta}$ in the multiplier $\boldsymbol{\lambda}=[\boldsymbol{\alpha};\boldsymbol{\beta}]\in \mathbb{R}^{(m+n)p}$, which is therefore associated with the constraint $A\mathbf {x}+B\mathbf {y}= 0$. Thus, the Lagrangian function of \eqref{compact_problem} is given by
\begin{equation*}
    \mathcal{L}(\mathbf{x}, \mathbf{y}, \boldsymbol{\lambda}) \triangleq F(\mathbf{x}) + H(\mathbf{y}) - \langle \boldsymbol{\lambda}, A\mathbf{x} + B\mathbf{y} \rangle .
\end{equation*}
Simultaneously, We define the augmented Lagrangian function of \eqref{compact_problem} as
\begin{align} \label{AL_fun_compact}
    \mathcal{L}_{\rho}(\mathbf{x}, \mathbf{y}, \boldsymbol{\lambda}) &\triangleq F(\mathbf{x}) + H(\mathbf{y}) - \langle \boldsymbol{\lambda}, A\mathbf{x} + B\mathbf{y} \rangle \notag \\
    &\quad + \frac{\rho}{2} \|A\mathbf{x} + B\mathbf{y}\|^2,
\end{align}
where $\rho >0$ is a penalty parameter, which will be appropriately chosen in the subsequent analysis.
\subsection{Stochastic Recursive Momentum}
Directly optimizing \eqref{AL_fun_compact} via standard stochastic gradient algorithms leads to high variance and a sublinear convergence rate. To accelerate convergence, we introduce a stochastic gradient estimator $\mathbf{v}^k$ using the momentum technique\cite{STORM2019} to track the gradient $\nabla F(\mathbf{x})$
\begin{align} \label{update v compact}
\mathbf{v}^{k+1} = \nabla F(\mathbf{x}^{k+1}, \xi^{k+1}) + (1 - a^{k+1}) (\mathbf{v}^k - \nabla F(\mathbf{x}^k, \xi^{k+1})).
\end{align}
Specifically, for each agent $i$, the local estimator $v_i^k$ is updated as follows:
\begin{align}\label{vi_update}
v_i^{k+1} &=(1-a^{k+1})(v_i^k-\nabla f_i(x_i^k,\xi_i^{k+1}))\notag\\&\quad + \nabla f_i(x_i^{k+1},\xi_{i}^{k+1}),
\end{align}
where $a^{k+1}\in(0,1]$ is the momentum parameter, and $\xi_i^{k+1}$ is the data sample drawn at iteration $k+1$. This estimator reduces the variance of the stochastic gradient approximation, which is critical for nonconvex optimization.

\subsection{Primal-Dual Updates}
At iteration $k$, HSM-ADMM updates the variables in the order $\mathbf{y}^{k+1}$, $\mathbf{x}^{k+1}$, and $\boldsymbol{\lambda}^{k+1}$.

\textit{Update of} $\mathbf {y}$: We update the auxiliary variable $\mathbf{y}$ by minimizing the augmented Lagrangian function $\mathcal{L}_{\rho^k}$ with respect to $\mathbf{y}$ while fixing $\mathbf{x}^k$ and $\boldsymbol{\lambda}^k$. Omitting terms independent of $\mathbf{y}$, the problem reduces to:
\begin{align}\label{y_update_compact}
    \mathbf{y}^{k+1} &= \argmin_{\mathbf{y}} \left\{ H(\mathbf{y}) - \langle \boldsymbol{\lambda}^k, B\mathbf{y} \rangle + \frac{\rho^k}{2} \|A\mathbf{x}^k + B\mathbf{y}\|^2 \right\} \notag \\
    &=\argmin_{\mathbf{y}} \left\{ H(\mathbf{y}) + \langle \boldsymbol{\beta}^k, \mathbf{y} \rangle + \frac{\rho^k}{2} \|\mathbf{x}^k - \mathbf{y}\|^2 \right\} \notag \\
    &=\argmin_{\mathbf{y}} \left\{ H(\mathbf{y}) + \frac{\rho^k}{2} \left\| \mathbf{y} - \left( \mathbf{x}^k - \frac{\boldsymbol{\beta}^k}{\rho^k} \right) \right\|^2 \right\} \notag \\
    &= {\rm prox}_H^\frac{1}{\rho^k} \left( \mathbf{x}^k - \frac{\boldsymbol{\beta}^k}{\rho^k} \right).
\end{align}
Due to the separability of $H(\mathbf{y})$ and the diagonal structure of the constraints, this global update decomposes into $n$ independent local updates:
\begin{align}\label{yi_update}
    y_i^{k+1} = \mathrm{prox}_{h_i}^{\frac{1}{\rho^k}} \left( x_i^k - \frac{\beta_i^k}{\rho^k} \right).
\end{align}
This step requires no communication with neighbors, as it depends only on local variables $(x_i^k, \beta_i^k)$ and the local function $h_i$.

\textit{Update of} $\mathbf {x}$: With $\mathbf{y}^{k+1}$ available, we now minimize the augmented Lagrangian function $\mathcal{L}_{\rho^k}(\mathbf{x}, \mathbf{y}^{k+1}, {\boldsymbol{\lambda}}^k)$, given by
\begin{align}
\mathcal{L}_{\rho^k}(\mathbf{x}, \mathbf{y}^{k+1}, {\boldsymbol{\lambda}}^k) &= F(\mathbf{x}) + H(\mathbf{y}^{k+1}) - \langle {\boldsymbol{\lambda}}^k,A\mathbf{x} + B\mathbf{y}^{k+1} \rangle \notag \\
&\quad + \frac{\rho^k}{2}\| A\mathbf{x} + B\mathbf{y}^{k+1} \|^2 .   
\end{align}

However, directly minimizing $\mathcal{L}_{\rho^k}$ with respect to $\mathbf{x}$ is computationally expensive  due to the nonconvexity of $F(\mathbf{x})$ and the coupling in the quadratic penalty term $\|A\mathbf{x} + B\mathbf{y}^{k+1}\|^2$. To obtain an efficient update, we construct a surrogate function $\hat{\mathcal{L}}_{\rho^k}$ by linearizing the smooth part $F(\mathbf{x})$ and the quadratic penalty term $\frac{\rho^k}{2}\| A\mathbf{x} + B\mathbf{y}^{k+1} \|^2$ at $\mathbf{x}^k$, and adding a proximal term $\frac{1}{2}\|\mathbf{x} - \mathbf{x}^k\|_{Q^k}^2$ parameterized by the adaptive step-size matrix $Q^k$ to ensure robust convergence:
\begin{align}\label{surrogate function}
\hat{\mathcal{L}}_{\rho^k}(\mathbf{x}, \mathbf{y}^{k+1}, {\boldsymbol{\lambda}}^k) &= F(\mathbf{x}^k) +\langle \mathbf{v}^k,\mathbf{x} -\mathbf{x}^{k} \rangle +  \frac{1}{2}\|\mathbf{x}-\mathbf{x}^k\|_{Q^k}^2 \notag \\
&\quad + H(\mathbf{y}^{k+1}) - \langle {\boldsymbol{\lambda}}^k,A\mathbf{x} + B\mathbf{y}^{k+1} \rangle \notag \\
&\quad + \langle \rho^kA^{\mathrm{T}}(A\mathbf{x}^k + B\mathbf{y}^{k+1}),\mathbf{x}-\mathbf{x}^k \rangle \notag \\
&\quad + \frac{\rho^k}{2}\| A\mathbf{x}^k + B\mathbf{y}^{k+1} \|^2.
\end{align}
Here, $\mathbf{v}^k$ serves as a stochastic approximation of $\nabla F(\mathbf{x}^k)$.

The proximal matrix $Q^k$ is crucial for stability. A standard choice is $Q^k = \eta^k I$, which requires $\eta^k$ to be large enough to cover the spectral radius of the graph Laplacian (i.e., $\eta^k \propto d_{\max}$). This forces the entire network to update at the pace of the bottleneck node. To overcome this, we propose a heterogeneous adaptive step-size strategy:
\begin{align*}
    Q^k &= \mathrm{diag}(\eta_1^k I_p, \dots, \eta_n^k I_p), \\
    \eta_i^k &= c_\eta (d_i + 1) k^{1/3},
\end{align*}
where $c_\eta > 0$ is a topology-independent constant. We can also express $Q^k$ as
\begin{align*}
Q^k = k^{1/3} \underbrace{\text{diag}\left( c_\eta(d_1+1)I_p, \dots, c_\eta(d_n+1)I_p \right)}_{\triangleq C_\eta }.
\end{align*}
By scaling  $\eta_i^k$ with the local degree $d_i$, we effectively counterbalance the scaling effect of the Laplacian term in the gradient, allowing each agent to update with a step size proportional to its local connectivity. This design decouples the algorithm's stability from the global maximum degree, enabling faster convergence in heterogeneous networks.

Ignoring terms independent of $\mathbf{x}$, problem \eqref{surrogate function} reduces to
\begin{align}\label{eq:X_min}
\mathbf{x}^{k+1} &= \argmin_{\mathbf{x}} \biggl( \langle \mathbf{v}^k - A^\mathrm{T} {\boldsymbol{\lambda}}^k + \rho^k A^\mathrm{T} (A\mathbf{x}^k + B\mathbf{y}^{k+1}),\mathbf{x} \rangle \notag \\
&\quad + \frac{1}{2}\|\mathbf{x}-\mathbf{x}^k\|_{Q^k}^2\biggr).    
\end{align}
Using the first-order optimality condition of \eqref{eq:X_min}, we have
\begin{align} \label{x_optimality}
    \mathbf{v}^k - A^\mathrm{T} \boldsymbol{\lambda}^k + \rho^k A^\mathrm{T} (A\mathbf{x}^k + B\mathbf{y}^{k+1}) + Q^k (\mathbf{x}^{k+1} - \mathbf{x}^k) = 0,
\end{align}
thus,  $\mathbf{x}^{k+1}$ is updated as
\begin{align} \label{x_update_compact}
    \mathbf{x}^{k+1} = \mathbf{x}^k - (Q^k)^{-1} \left( \mathbf{v}^k - A^\mathrm{T} \boldsymbol{\lambda}^k + \rho^k A^\mathrm{T}(A\mathbf{x}^k + B\mathbf{y}^{k+1}) \right).
\end{align}
Extracting the $i$-th component from \eqref{x_update_compact}, we obtain the corresponding local update for each agent $i$:
\begin{align}\label{xi_update}
x_{i}^{k+1} &=x_{i}^k - \frac{1}{\eta_i^k} 
\Bigl( v_i^k + \sum_{j \in \mathcal{N}_i} [-\alpha_{ij}^k + \rho^k(x_i^k - x_j^k)] \notag \\
&\quad -\beta_i^k + \rho^k(x_i^k - y_i^{k+1}) \Bigr).
\end{align}
Obviously, this step involves communication only with immediate neighbors $j \in \mathcal{N}_i$.

\textit{Update of} $\bm{\lambda}$: The dual variables $\bm{\lambda}=[\boldsymbol{\alpha};\boldsymbol{\beta}]$ are updated via gradient ascent on the augmented Lagrangian function:
\begin{align} \label{update lambda compact}
\bm{\lambda}^{k+1}= \bm{\lambda}^{k} -\rho^k \left ( A\mathbf {x}^{k+1} +B\mathbf {y}^{k+1}\right ).
\end{align}

\begin{algorithm}[t]\label{DSMADMM}
\caption{\textbf{HSM-ADMM} ~~--- from the view of agent $i$}
\begin{algorithmic}[1]
\vspace{0.1cm}
\State \textbf{Initialization:}$x_i^0,y_i^0,\alpha_{ij}^0,\beta_i^0$.
\vspace{0.1cm}
\State \textbf{Parameters:} Set $\rho^k = c_\rho k^{1/3}$, $a^k = c_a k^{-2/3}$, and $\eta_i^k = c_\eta (d_i + 1) k^{1/3}$.
\vspace{0.1cm}
\State{let $v_i^0=\frac{1}{m^0}\sum\limits_{s=1}^{m^0}\nabla f_i(x_i^0,\xi_{i,s}^{0})$.}
\vspace{0.1cm}
\For{$k = 0, \ldots, K-1$}
\vspace{0.1cm}
\State{Update $y_i^{k+1}$ and $x_i^{k+1}$:}
\State{$\quad y_i^{k+1}={\rm prox}_{h_i}^{\frac{1}{\rho^k}} \left( x_i^k - \frac{\beta_i^k}{\rho^k} \right)$.}
\vspace{0.1cm}
\State{$\quad x_i^{k+1}=x_{i}^k - \frac{1}{\eta_i^k} 
\bigl ( v_i^k + \sum_{j \in \mathcal{N}_i} [-\alpha_{ij}^k + \rho^k(x_i^k - x_j^k)] $}
\Statex{$\quad \quad \quad \quad \quad \quad -\beta_i^k + \rho^k(x_i^k - y_i^{k+1}) \bigr)$.} 
\vspace{0.1cm}
\State{Receive $x_j^{k+1}$ from neighbors $j\in \mathcal{N}_{i}$.}
\vspace{0.1cm}
\State{Update $\alpha_{ij}^{k+1}$ and $\beta_i^{k+1}$:}
\State{$\quad \alpha_{ij}^{k+1}=\alpha_{ij}^{k}-\rho^k\left( x_i^{k+1}-x_j^{k+1}\right)$.}
\vspace{0.1cm}
\State{$\quad \beta_i^{k+1}=\beta_i^{k}-\rho^k\left( x_i^{k+1}-y_i^{k+1}\right)$.}
\vspace{0.1cm}
\State Sample $\xi_i^{k+1}\in \mathcal{D}_i^k$ and let
\vspace{0.1cm}
\Statex{$\quad \quad \quad \quad v_i^{k+1} = (1-a^{k+1})(v_i^k-\nabla f_i(x_i^k,\xi_i^{k+1}))$}
\vspace{0.1cm}
\Statex{$\quad \quad \quad \quad \quad \quad \quad + \nabla f_i(x_i^{k+1},\xi_{i}^{k+1})$.} 
\vspace{0.1cm}
\EndFor
\end{algorithmic}
\end{algorithm}
We summarize the full procedure of the proposed HSM-ADMM in Algorithm 1. Before proceeding to the rigorous theoretical analysis, we highlight three appealing features that distinguish our algorithm from existing works from an implementation perspective:

\indent (1) \textbf{Topology-Aware Adaptability:} The most distinguishing feature of HSM-ADMM lies in the heterogeneous adaptive step-size design. Unlike existing distributed algorithms (e.g., SPPDM \cite{Wang2021TSP_SPPDM}, DEEPSTORM \cite{AAAIDEEPSTORM}) that enforce a uniform step-size $\eta^k$ strictly bounded by the global maximum degree $d_{\max}$ or the network spectral radius, HSM-ADMM assigns a specific step-size $\eta_i^k = c_\eta(d_i+1)k^{1/3}$ to each agent $i$. This design successfully decouples the local update pace from global network bottlenecks, allowing agents in sparsely connected regions to take larger, more aggressive steps.

\indent (2) \textbf{Communication Efficiency:} Communication overhead is often the primary bottleneck in distributed systems. State-of-the-art gradient tracking methods (e.g., ProxGT-SA \cite{xin2021stochastic}) require transmitting both the primal variable $x_i$ and an auxiliary gradient tracking variable $s_i$ at each iteration. In contrast, HSM-ADMM requires agents to broadcast only their primal variable $x_i^{k+1}$ to their immediate neighbors. This structural advantage effectively halves the bandwidth consumption per iteration.

\indent (3) \textbf{Computational and Memory Efficiency:} HSM-ADMM operates in a strictly single-loop manner. It evaluates only a single stochastic gradient (or a mini-batch of size $\mathcal{O}(1)$) per iteration to update the momentum estimator $v_i^{k+1}$ \eqref{vi_update}, bypassing the need for periodic full-gradient evaluations required by SVRG/SPIDER-based algorithms. Furthermore, the $y_i$-update \eqref{yi_update} solely relies on a standard proximal operator, and the $x_i$-update \eqref{xi_update} is highly parallelizable and computationally light. The memory requirement per agent is also minimal, as each agent only needs to store its local variables $(x_i, y_i, \beta_i, v_i)$ and the dual variables $\alpha_{ij}$ corresponding to its incident edges.

\section{Convergence Analysis}\label{sec-IV}
In this section, we establish the theoretical convergence guarantees for the proposed HSM-ADMM. We first introduce the optimality measure used to quantify the algorithm's performance. Next, we provide several key supporting lemmas. Building upon these foundations, we finally derive the explicit global convergence rate of our algorithm.

\subsection{Optimality Condition and Stationarity}
The convergence of the algorithm is evaluated based on the Karush-Kuhn-Tucker (KKT) conditions. A point $(\mathbf{x}^\star, \mathbf{y}^\star, \boldsymbol{\lambda}^\star)$ is deemed a stationary point if $\mathbf{0} \in \partial \mathcal{L}(\mathbf{x}^\star, \mathbf{y}^\star, \boldsymbol{\lambda}^\star)$. Given the stochastic nature of our setting, we adopt the expected stationarity gap as the standard convergence metric.

\begin{definition}[$\epsilon$-stationary Point]\label{def:epsilon_stationary_point}
For any given $\epsilon > 0$, a point $(\mathbf{x}^\star, \mathbf{y}^\star, \boldsymbol{\lambda}^\star)$ is said to be an $\epsilon$-stationary point of problem \eqref{compact_problem} if it satisfies:
\begin{align}\label{eq:stationary_def}
\mathbb{E}\left[ \mathrm{dist}^2\left( 0, \partial \mathcal{L}(\mathbf{x}^\star, \mathbf{y}^\star, \boldsymbol{\lambda}^\star) \right) \right] \leq \epsilon,
\end{align}
where $\mathrm{dist}(0, \mathcal{S}) \triangleq \min_{\mathbf{g} \in \mathcal{S}} \|\mathbf{g}\|$, and the subdifferential of the Lagrangian function is given by:
\begin{align}\label{subdifferential_of_Lagrangian}
\partial \mathcal{L}\left( \mathbf {x},\mathbf {y},\bm{\lambda} \right) &=\left[ \begin{array}{c}
	\nabla _\mathbf {x}\mathcal{L}\left( \mathbf {x},\mathbf {y},\bm{\lambda} \right)\\
	\partial _\mathbf {y}\mathcal{L}\left( \mathbf {x},\mathbf {y},\bm{\lambda} \right)\\
	\nabla _{\bm{\lambda}}\mathcal{L}\left( \mathbf {x},\mathbf {y},\bm{\lambda} \right)\\
\end{array} \right] \notag
\\&=\left[ \begin{array}{c}
	\nabla F\left( \mathbf {x} \right)  -A^{\mathrm{T}}\bm{\lambda}\\
	\partial H\left( \mathbf {y} \right)-B^{\mathrm{T}}\bm{\lambda}\\
	A\mathbf {x}+B\mathbf {y}\\
\end{array} \right]. 
\end{align}
\end{definition}
\subsection{Key Lemmas}
First, we establish an upper bound on the difference between consecutive dual variables, which is crucial for handling the coupling introduced by the network constraints. 
\begin{lemma}\label{Upper bound of dual variable error}
Suppose that Assumptions \ref{Ass: undirected and connected}, \ref{Ass: mean-squared smoothness}, \ref{Ass: f,g lower bounded}, \ref{Ass: unbiasedness and variance boundedness} hold. Let  $\left ( \mathbf {x}^k,\mathbf {y}^k,\bm{\lambda}^k \right ) $ be the sequence generated by Algorithm 1. For any constant $\theta > 0$, the squared norm of the dual variable difference is bounded by:
\begin{align}
&\|\Delta {\boldsymbol{\lambda}}^{k+1}\|^2 \notag\\
\le&   (1+\theta) (\Delta \mathbf{x}^{k+1})^\mathrm{T} (S^k)^\mathrm{T} S^k (\Delta \mathbf{x}^{k+1}) \notag \\
+& \left ( 2(1+\frac{1}{\theta})\|S^{k-1}\|_2^2 + 4L^2(1+\frac{1}{\theta}) \right ) \|\Delta \mathbf{x}^k\|^2 \notag \\
+& 8(1+\frac{1}{\theta})\|\mathcal{E}^k \|^2+8(1+\frac{1}{\theta})\| \mathcal{E}^{k-1}\|^2,
\end{align}
where $\Delta \boldsymbol{\lambda}^{k+1} \triangleq \boldsymbol{\lambda}^{k+1} - \boldsymbol{\lambda}^k$, $\Delta \mathbf{x}^{k+1} \triangleq \mathbf{x}^{k+1} - \mathbf{x}^{k}$, the auxiliary matrix is defined as $S^k \triangleq Q^k - \rho^k A^\mathrm{T} A$, and $\mathcal{E}^k \triangleq \mathbf{v}^k - \nabla F(\mathbf{x}^k)$ is the gradient estimation error. Note that we have utilized the topological property $\sigma_{\min}^2(A)=1$ to simplify the bound.
\end{lemma}

\textit{Proof:} See Appendix~\ref{appendix: proof of Upper bound of dual variable error}.
\begin{remark}
We deliberately retain the matrix quadratic form $(\Delta \mathbf{x}^{k+1})^\mathrm{T} (S^k)^\mathrm{T} S^k (\Delta \mathbf{x}^{k+1})$ for the current iteration instead of relaxing it to a scalar norm (e.g., $\|S^k\|_2^2$). This tight characterization allows the adaptive step-size matrix $Q^k$ to locally neutralize the topology-induced ill-conditioning in the subsequent descent analysis. 
\end{remark}

Next, we characterize the variance reduction property of the recursive momentum estimator.
\begin{lemma}\label{Recursive Bound on Gradient Estimation Error}
Under Assumptions \ref{Ass: mean-squared smoothness}, the estimation error of the STORM estimator satisfies the following ecursive bound:
\begin{align}
\mathbb{E}[\|\mathcal{E}^{k+1}\|^2] \le& (1 - a^{k+1})^2 \mathbb{E}[\|\mathcal{E}^k\|^2] + 2 ( a^{k+1} )^2 \sigma^2 \notag \\&+ 2L^2  (1- a^{k+1})^2 \mathbb{E}[\|\Delta \mathbf{x}^{k+1}\|^2],
\end{align}
where $\sigma^2$ is the bound on the stochastic gradient variance.
\end{lemma}

\textit{Proof:} See Appendix~\ref{appendix: Recursive Bound on Gradient Estimation Error}.
\subsection{Lyapunov Function and Descent Property}
To establish the global convergence, we construct the Lyapunov function as follows:
\begin{align}\label{eq: Lyapunov function}
\Phi^k &\triangleq  \mathcal{L}_{\rho^{k-1}}(\mathbf{x}^k, \mathbf{y}^k, {\boldsymbol{\lambda}}^k) + \frac{1}{\Gamma^k} \|\mathcal{E}^k\|^2 \notag \\&\quad+ \frac{C_{err}}{\rho^{k-1} }\|\mathcal{E}^{k-1}\|^2 + \frac{\beta^k}{2} \|\Delta \mathbf{x}^k\|^2,
\end{align}
where the parameters $\Gamma^{k}$, $C_{err}$ and $\beta^k$ satisfy
\begin{align}\label{Gamma^k}
\frac{1}{\Gamma^{k}}  =c_\gamma k^{1/3},    
\end{align}
\begin{align}\label{C_err}
C_{err} \ge 12(1+\frac{1}{\theta}),
\end{align}
\begin{align}\label{beta^k}
\beta^k \ge \left(\frac{6\left(1+\frac{1}{\theta}\right)  \| C_\eta  - c_\rho A^\mathrm{T}A \|^2 }{c_{\rho} }+\frac{ 12L^2\left(1+\frac{1}{\theta}\right)}{c_{\rho} } \right) k^{\frac{1}{3}}.
\end{align}
\begin{lemma}\label{Single-Step Lyapunov Descent Lemma}
Suppose Assumptions \ref{Ass: undirected and connected}, \ref{Ass: mean-squared smoothness}, \ref{Ass: f,g lower bounded}, \ref{Ass: unbiasedness and variance boundedness} hold. Let $\{\mathbf{x}^k, \mathbf{y}^k, \boldsymbol{\lambda}^k\}$ be generated by Algorithm 1 with the adaptive step-size matrix $Q^k$. The sequence $\{\Phi^k\}$ satisfies the following single-step recursive inequality:
\begin{align}\label{Single-Step Lyapunov Descent}
    &\Phi^{k+1} - \Phi^k \notag\\
    \le & - (\Delta \mathbf{x}^{k+1})^{\mathrm{T}} C_{\mathbf{x}}^k (\Delta \mathbf{x}^{k+1}) - C_{\mathcal{E}}^k \|\mathcal{E}^k\|^2 \notag \\
    & - \frac{\rho^k}{2} \|\mathbf{r}^{k+1}\|^2 + \frac{\rho^k - \rho^{k-1}}{2} \|\mathbf{r}^k\|^2  + \frac{2(a^{k+1})^2 \sigma^2}{\Gamma^{k+1}},
\end{align}
where 
\begin{align*}
C_{\mathbf{x}}^k &= \left( Q^k - \frac{\rho^k}{2} A^\mathrm{T} A \right) - \frac{3(1+\theta)}{2\rho^k } (S^k)^\mathrm{T} S^k \\
&\quad - \left( \frac{\mu^k}{2} + \frac{\beta_{k+1}}{2} + \frac{L}{2} + \frac{2L^2(1-a^{k+1})^2}{\Gamma^{k+1}} \right) I,\\
C_{\mathcal{E}}^k &= \frac{1}{\Gamma^k} -\frac{(1-a^{k+1})^2}{\Gamma^{k+1}}  - \frac{1}{2\mu^k} - \frac{12(1+\frac{1}{\theta})}{\rho^k } - \frac{C_{err}}{\rho^k},\\
\mathbf{r}^{k} &= A\mathbf{x}^{k} + B\mathbf{y}^{k}.
\end{align*}
\end{lemma}

\textit{Proof:} See Appendix~\ref{appendix: Single-Step Lyapunov Descent Lemma}.
\subsection{Global Convergence Rate}
Based on the descent property established above, we now provide the global convergence guarantees.
\begin{theorem}\label{potential function accumulation}
It follows from Assumption \ref{Ass: f,g lower bounded} that there exists a universal lower bound $\Phi^{\star}$ for the sequence $\{\Phi^k\}$. Summing \eqref{Single-Step Lyapunov Descent} from $k=1$ to $K$ yields the following finite accumulation bound:
\begin{align}\label{eq:accumulation bound}
&\sum_{k=1}^K \left (  k^{-1/3} \|\mathcal{E}^k\|^2 + k^{1/3} \|\Delta \mathbf{x}^{k+1}\|^2 + k^{1/3} \|\mathbf{r}^{k+1}\|^2 \right ) \notag \\ &\le \frac{\Phi^{1} - \Phi^{\star} + \frac{\rho^1 -\rho^0}{2}\|\mathbf{r}^1\|^2+ 2\sigma^2c_a^2c_\gamma\left (1+ \ln{K} \right )}{\min \left (C_{\mathcal{E}},\lambda_{\min}(C_\mathbf{x}),\frac{c_\rho}{4} \right )  },
\end{align}
where the strictly positive constants are defined as:
\begin{align*}
C_{\mathcal{E}}&= 2c_a c_\gamma  - \frac{1}{2c_\mu} - \frac{12(1+\frac{1}{\theta})}{c_\rho } - \frac{C_{err}}{c_\rho}, \\
C_\mathbf{x} &= \left( C_\eta - \frac{c_\rho}{2} A^\mathrm{T} A \right) \\&\quad - \frac{3(1+\theta)}{2c_\rho } (C_\eta - c_\rho A^\mathrm{T} A)^\mathrm{T} (C_\eta - c_\rho A^\mathrm{T} A)\\
&\quad - \left( \frac{c_\mu}{2} + \frac{c_\beta}{2} + \frac{L}{2} +2L^2 c_\gamma \right) I.
\end{align*}
\end{theorem}

\textit{Proof:} See Appendix~\ref{appendix: potential function accumulation}.
\begin{remark}[Asymptotic Regularity]
Theorem \ref{potential function accumulation} establishes the global asymptotic stability of the proposed HSM-ADMM. Specifically, the finiteness of the weighted sum implies that $\liminf_{k \to \infty} \mathbb{E}\bigl[\|\mathcal{E}^k\|^2 + \|\Delta \mathbf{x}^{k+1}\|^2 + \|\mathbf{r}^{k+1}\|^2\bigr] = 0.$. This guarantees that as the iteration progresses, the gradient estimation error diminishes, the primal variable sequence stabilizes, and the consensus constraints are asymptotically satisfied in expectation, despite the inherent variance of stochastic gradients and the nonconvexity of the objective function.
\end{remark}

\begin{theorem}[Optimal Convergence Rate]\label{Convergence rate theorem}
Suppose that Assumptions \ref{Ass: undirected and connected}, \ref{Ass: mean-squared smoothness}, \ref{Ass: f,g lower bounded}, \ref{Ass: unbiasedness and variance boundedness} hold. Let  $\left ( \mathbf {x}^k,\mathbf {y}^k,\bm{\lambda}^k \right ) $ be the sequence generated by Algorithm 1 with $\rho^k=c_\rho k^{1/3}$, $a^k=c_a k^{-2/3}$. Then, the expected stationarity gap is bounded by:
\begin{align}
\min_{1 \leq k \leq K}\mathbb{E}\left[ \text{dist}^2\left( 0,\partial \mathcal{L}\left( \mathbf {x}^k,\mathbf {y}^k,\bm{\lambda} ^k \right) \right) \right] \leq \tilde{\mathcal{O}}\left ( K^{ -\frac{2}{3}} \right ),
\end{align}
achieving the optimal rate of $\tilde{\mathcal{O}}(K^{-2/3})$ (or $\tilde{\mathcal{O}}(\epsilon^{-1.5})$ in terms of stochastic oracle complexity).
\end{theorem}

\textit{Proof:} See Appendix~\ref{appendix: Convergence rate theorem}.
\begin{remark}[Optimality and Topology Independence]
Theorem \ref{Convergence rate theorem} demonstrates that HSM-ADMM achieves the convergence rate of $\tilde{\mathcal{O}}(K^{-2/3})$, which precisely matches the optimal lower bound for first-order stochastic nonconvex optimization algorithms. Crucially, unlike existing centralized or consensus-based algorithms that rely on uniform step-sizes tied to the global spectral radius (e.g., $\eta < 1/\lambda_{\max}(\mathbf{L}_G)$), our result is derived under a heterogeneous step-size regime. This confirms that the proposed topology-aware adaptation strategy effectively mitigates the straggler effect without compromising the optimal convergence order.
\end{remark}

\section{Numerical Examples}\label{sec-V}
In this section, we validate the convergence theory of our algorithm and demonstrate its effectiveness over representative distributed algorithms, including SPPDM \cite{Wang2021TSP_SPPDM}, ProxGT-SR-O \cite{xin2021stochastic}, and DEEPSTORM \cite{AAAIDEEPSTORM}. All algorithms compared are implemented in Python using PyTorch and mpi4y.

Following the general settings in \cite{AAAIDEEPSTORM} and \cite{TesiXiao2023Prox-DASA}, we consider solving the classification problem
\begin{align*}
    \underset{x\in\mathbb{R}^p}{\min}~ \frac{1}{n} \sum_{i=1}^{n} \frac{1}{|\mathcal{D}_i|}\sum_{(a, b)\in \mathcal{D}_i}\ell_i(g(x,a),b) + l_1 \|x\|_1,
\end{align*}
on the a9a and MNIST datasets. Here $g(x,a)$ is the output of a neural network with parameter $x$ on data $a$, and $\ell_i$ is the cross-entropy loss function measuring the discrepancy between the output and the true label $b$. The local training dataset $\mathcal{D}_i$ is kept strictly private and is only accessible to agent $i$. The $\ell_1$-norm regularization term is incorporated to impose a sparsity structure on the network weights, with the regularization strength set to $\lambda = 0.0001$ based on standard practice.
\begin{figure}[h]
  \centering
  \begin{subfigure}{0.4\columnwidth}
    \centering
    \includegraphics[width=\linewidth]{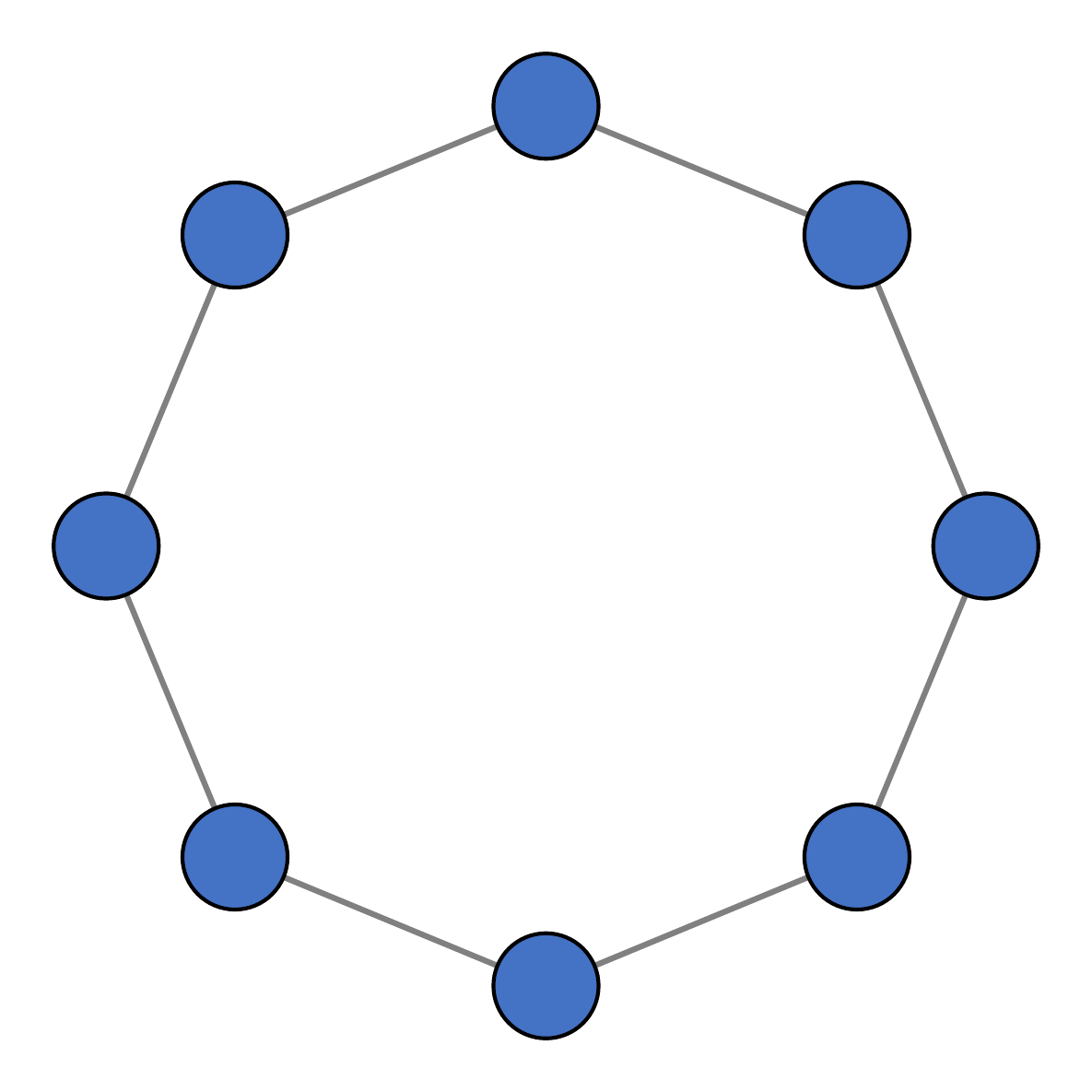}
    \caption{Ring graph}
    \label{fig: ring graph}
  \end{subfigure}
  \hfill
  \begin{subfigure}{0.4\columnwidth}
    \centering
    \includegraphics[width=\linewidth]{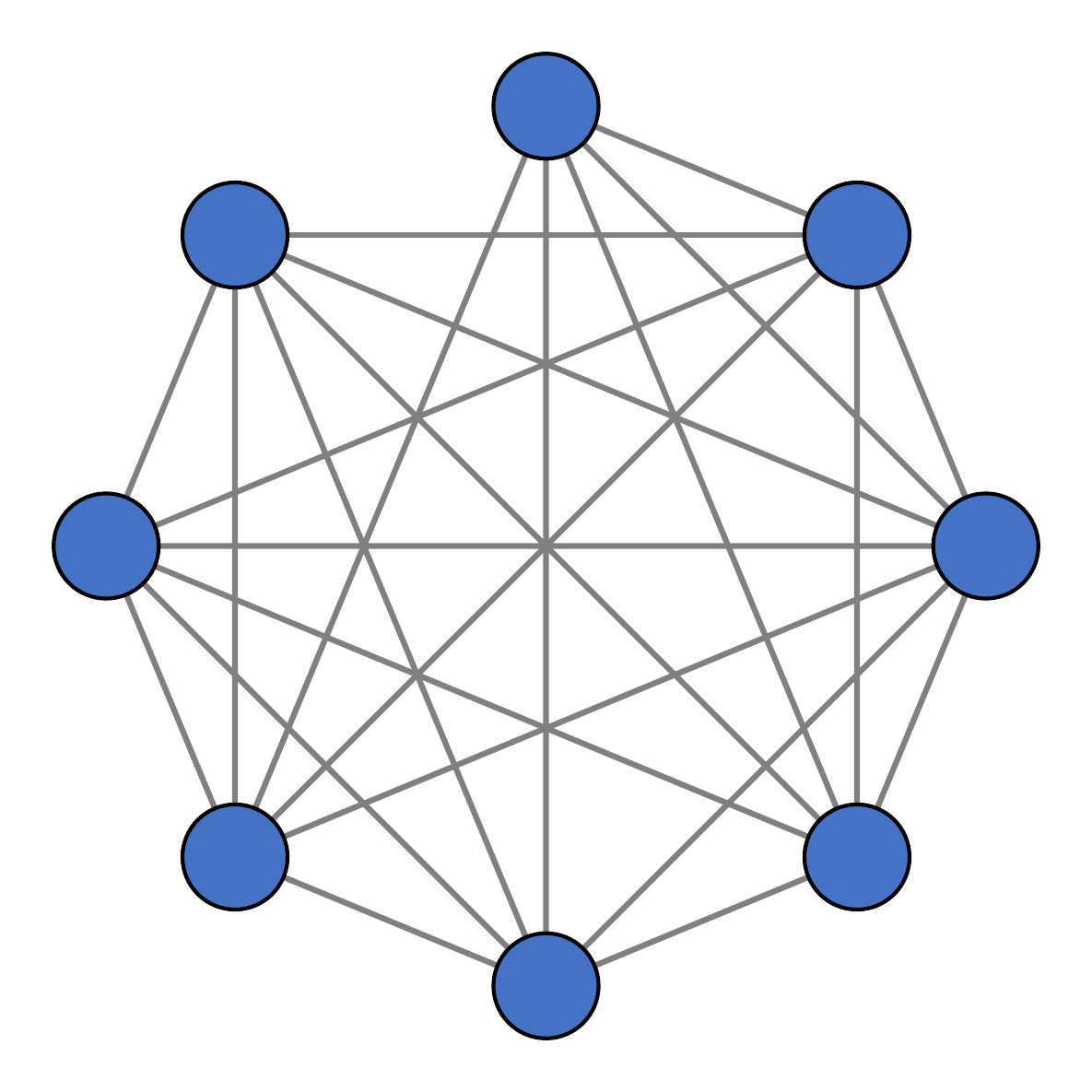}
    \caption{Random graph}
    \label{fig: random graph}
  \end{subfigure}
  \caption{Network topologies with 8 nodes}
  \label{fig: network topologies}
\end{figure}

\begin{figure*}[!t]
  \centering
  \begin{subfigure}{0.32\textwidth}
    \centering
    \includegraphics[width=\linewidth]{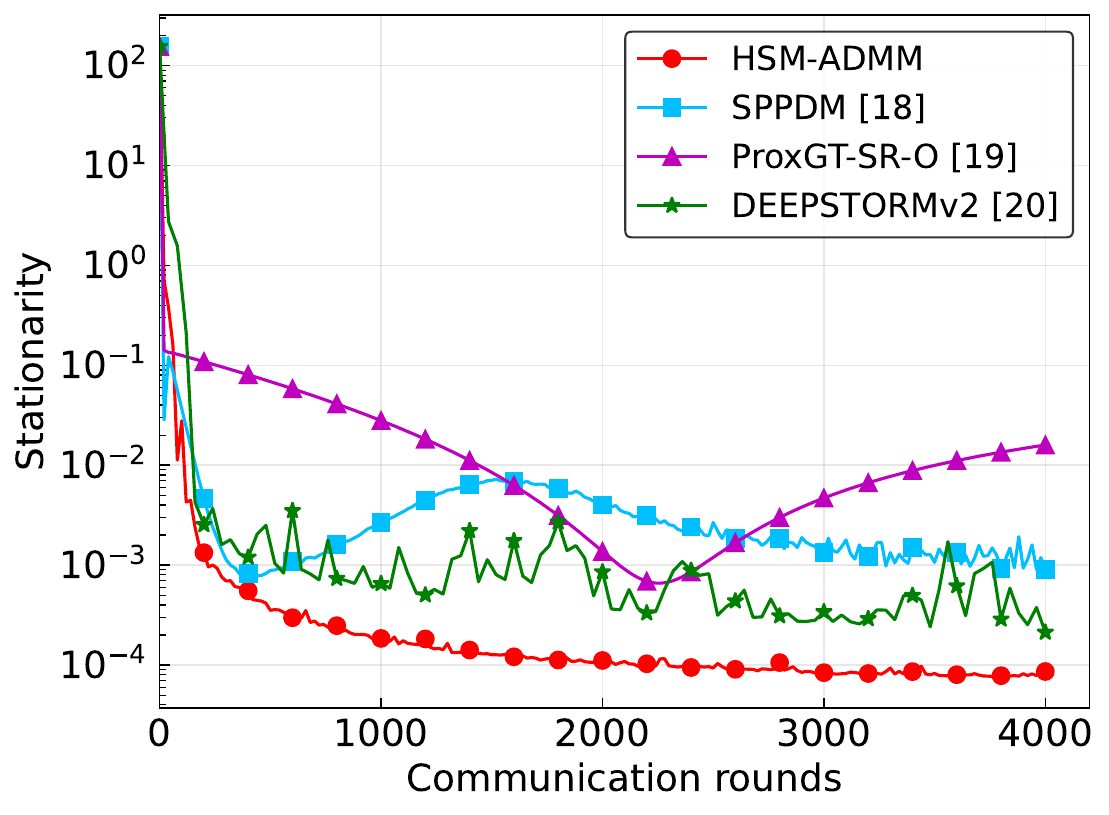}
    \caption{}
    \label{fig:ring_a9a_(a)}
  \end{subfigure}
  \hfill
  \begin{subfigure}{0.32\textwidth}
    \centering
    \includegraphics[width=\linewidth]{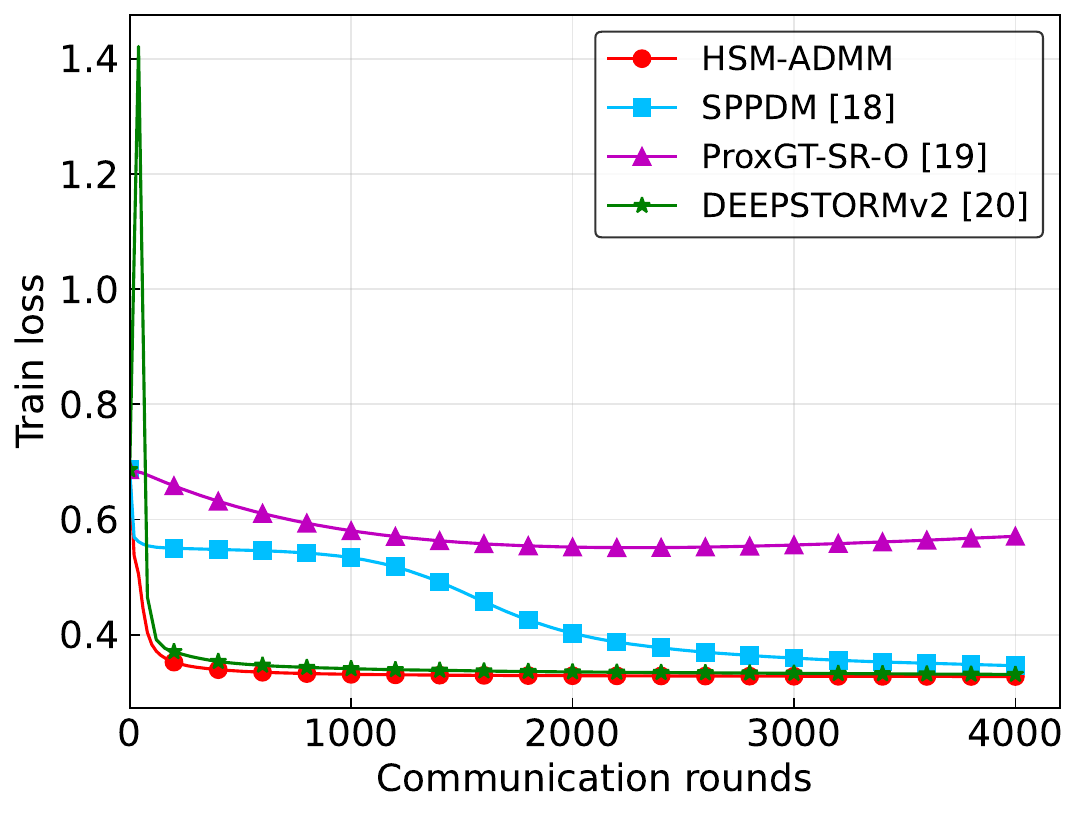}
    \caption{}
    \label{fig:ring_a9a_(b)}
  \end{subfigure}
  \hfill
  \begin{subfigure}{0.32\textwidth}
    \centering
    \includegraphics[width=\linewidth]{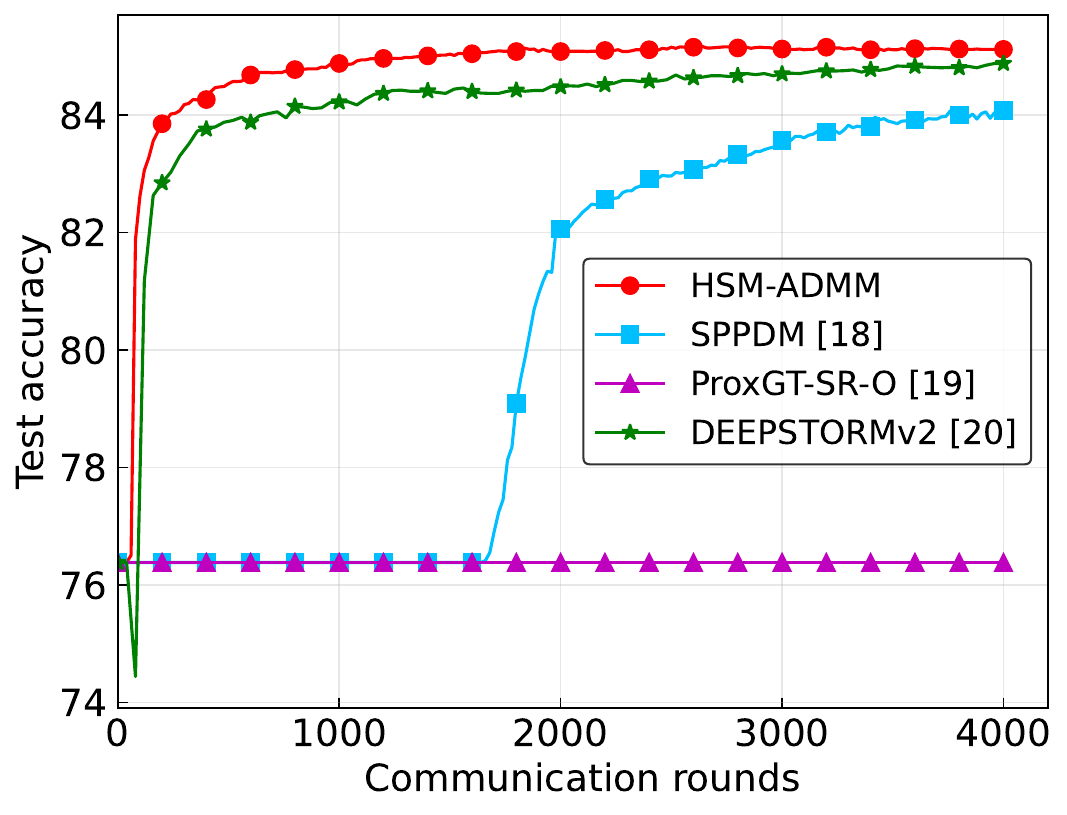}
    \caption{}
    \label{fig:ring_a9a_(c)}
  \end{subfigure}
  \caption{Performance comparison of distributed algorithms on a9a over a ring topology.}
  \label{fig: ring_a9a}
\end{figure*}

\begin{figure*}[!t]
  \centering
  \begin{subfigure}{0.32\textwidth}
    \centering
    \includegraphics[width=\linewidth]{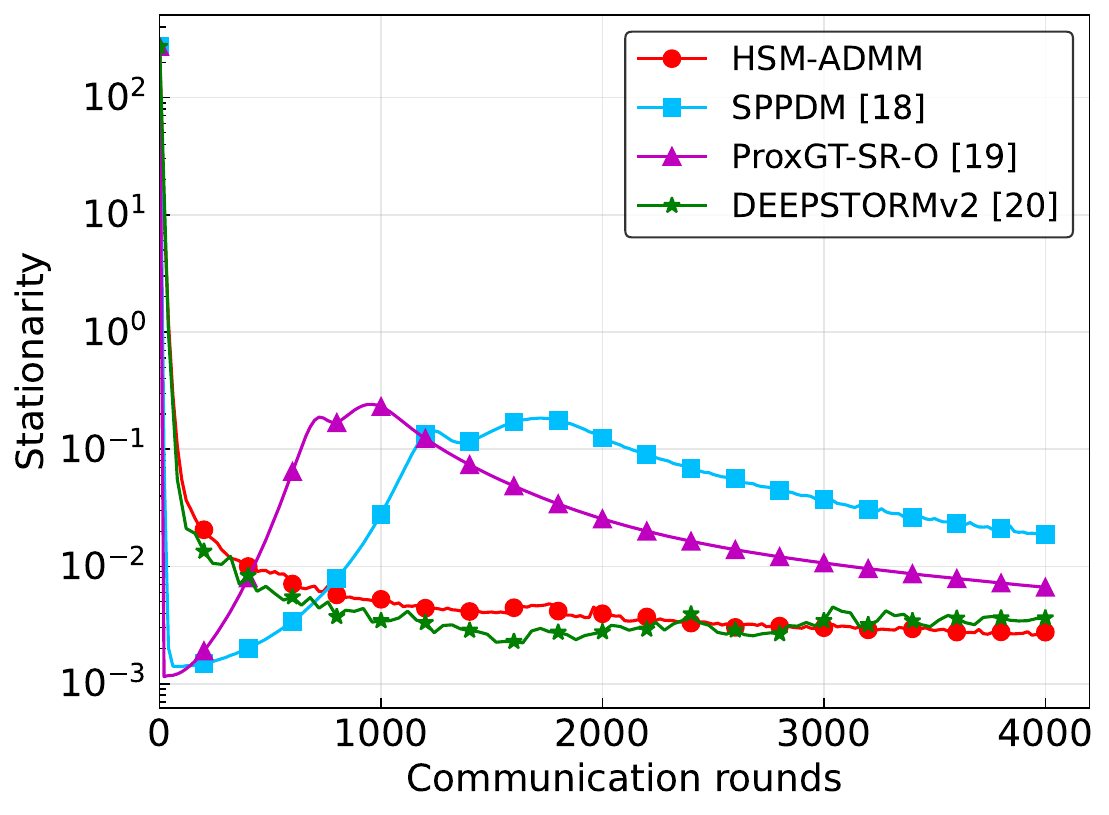}
    \caption{}
    \label{fig:MNIST_random_(a)}
  \end{subfigure}
  \hfill
  \begin{subfigure}{0.32\textwidth}
    \centering
    \includegraphics[width=\linewidth]{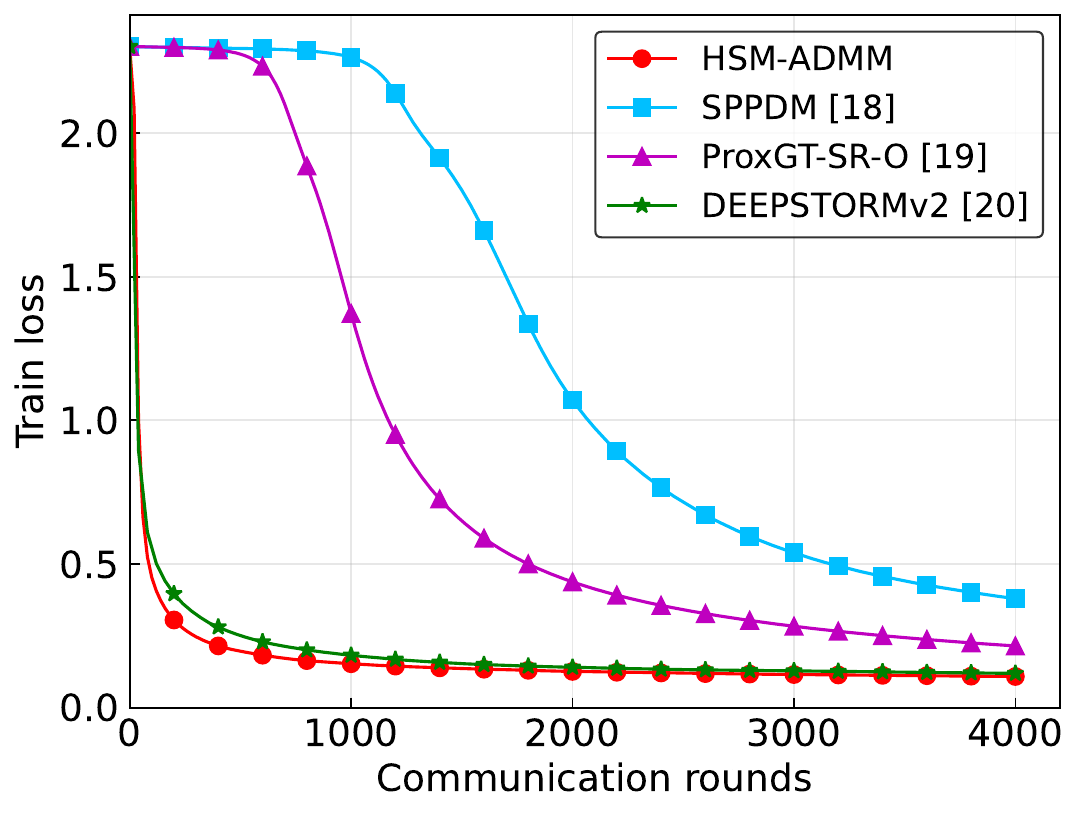}
    \caption{}
    \label{fig:MNIST_random_(b)}
  \end{subfigure}
  \hfill
  \begin{subfigure}{0.32\textwidth}
    \centering
    \includegraphics[width=\linewidth]{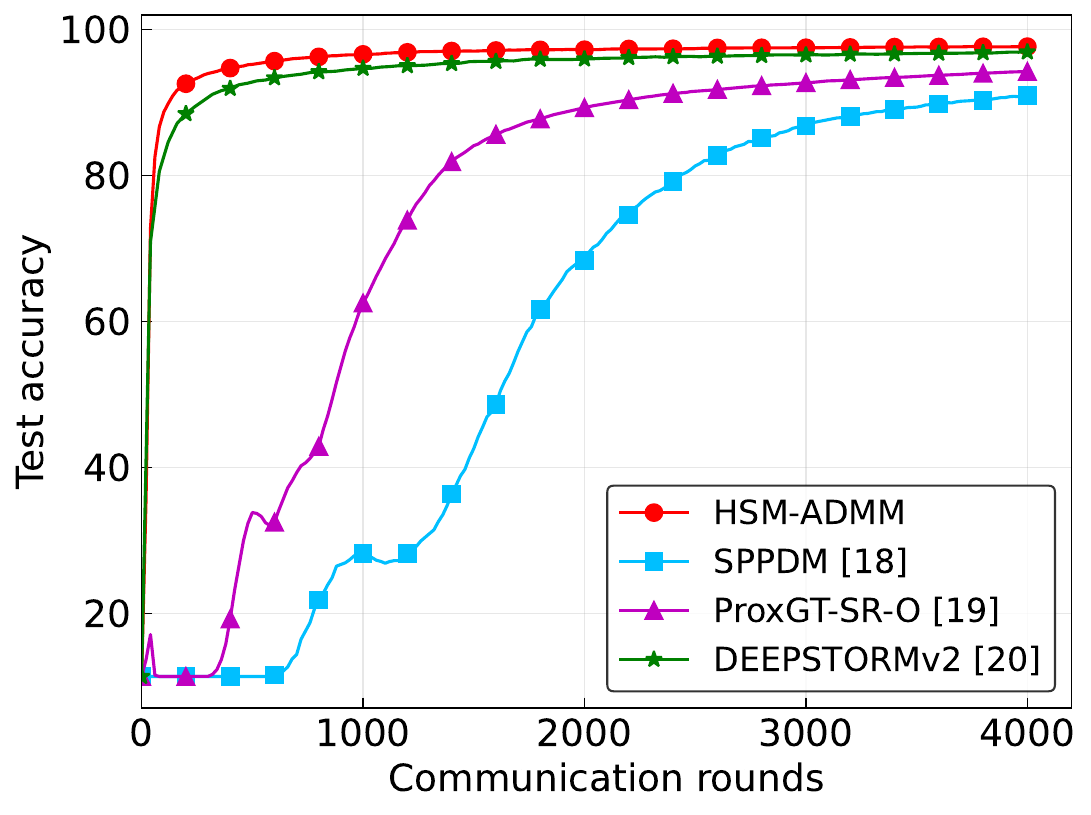}
    \caption{}
    \label{fig:MNIST_random_(c)}
  \end{subfigure}
  \caption{Performance comparison of distributed algorithms on MNIST over a random topology.}
  \label{fig: random_MNIST}
\end{figure*}
For the model architectures, we employ a 2-layer Multilayer Perceptron (MLP) for the a9a dataset and the standard LeNet architecture for the MNIST dataset. The communication network consists of 8 agents ($n=8$), which are connected over a ring topology for a9a and a randomly generated connected graph for MNIST, as illustrated in Fig.~\ref{fig: network topologies}. To demonstrate the performance of our algorithms in the constant batch size setting, the batch size is fixed to be 64 for a9a and 32 for MNIST for all algorithms. Furthermore, all algorithms are restricted to exactly one communication round per iteration to ensure a fair comparison of communication efficiency.

To quantitatively evaluate the convergence behavior, the stationarity measure at iteration $k$ is defined as follows:
\[\left\|\bar{x}^{k}-\hbox{prox}_{\sum_{i=1}^n{  h_i }}^1\left(\bar{x}^{k}-\nabla F(\bar{x}^{k})\right)\right\|_2^2 + \sum_{i=1}^n\left\|x_i^k-\bar{x}^k\right\|_2^2,\]
where $\bar{x}^k=\frac{1}{n}\sum_{i=1}^nx_i^k$, $ \nabla F(\bar{x}^k) = \frac{1}{n} \sum_{i=1}^n \nabla f_i(\bar{x}^k)$.

We periodically evaluate the performance metrics of these algorithms during the training phase. The results are plotted in Fig.~\ref{fig: ring_a9a} and Fig.~\ref{fig: random_MNIST}. It can be clearly observed that the proposed HSM-ADMM algorithm achieves superior performance in terms of the stationarity gap decay, training loss reduction, and overall test accuracy.

\section{Conclusions}\label{sec-VI}
In this paper, we propose a novel single-loop distributed algorithm, termed HSM-ADMM, to solve the nonconvex and nonsmooth stochastic  composite problem \eqref{Original Problem}. To address the scalability limitations of existing algorithms in highly heterogeneous networks, we introduce a heterogeneous adaptive step-size strategy. Our theoretical analysis establishes that this design successfully decouples the algorithmic stability from the global network topology (specifically, the spectral radius of the Laplacian matrix), thereby enabling robust convergence across diverse graph structures. Furthermore, by integrating the STORM recursive momentum estimator into a primal-dual framework, our algorithm achieves the optimal stochastic oracle complexity of $\tilde{\mathcal{O}}(\epsilon^{-1.5})$ using a constant $\mathcal{O}(1)$ mini-batch size. This allows HSM-ADMM to effectively handle stochastic noise without relying on large batches or cumbersome double-loop structures.

Future work includes extending this heterogeneous framework to asynchronous settings to handle communication delays, and incorporating gradient compression techniques to further reduce communication overhead in high-dimensional applications.

\section*{Appendix}


\renewcommand{\thesubsection}{\Alph{subsection}}
\subsection{Proof of Lemma \ref{Upper bound of dual variable error}}\label{appendix: proof of Upper bound of dual variable error}
Let $\tilde{\mathbf{r}}^k = A\mathbf{x}^k + B\mathbf{y}^{k+1}$ and $\mathbf{r}^{k+1} = A\mathbf{x}^{k+1} + B\mathbf{y}^{k+1}$, then
\[
\tilde{\mathbf{r}}^k = (A\mathbf{x}^{k+1} + B\mathbf{y}^{k+1}) - A(\mathbf{x}^{k+1} - \mathbf{x}^k) \notag = \mathbf{r}^{k+1} - A\Delta \mathbf{x}^{k+1}.
\]
According to the first-order optimality condition \eqref{x_optimality} for the $\mathbf{x}$-update, we have
\begin{align}\label{x_optimality_rewrite}
\mathbf{v}^k - A^\mathrm{T} {\boldsymbol{\lambda}}^k + \rho^k A^\mathrm{T} (\mathbf{r}^{k+1} - A\Delta \mathbf{x}^{k+1}) + Q^k \Delta \mathbf{x}^{k+1} = \mathbf{0}.
\end{align}
Moreover, the dual update rule \eqref{update lambda compact} implies:
\begin{align*}
A^\mathrm{T} {\boldsymbol{\lambda}}^{k+1} = A^\mathrm{T} {\boldsymbol{\lambda}}^k - \rho^k A^\mathrm{T} \mathbf{r}^{k+1}.
\end{align*}
Substituting $\tilde{\mathbf{r}}^k= \mathbf{r}^{k+1} - A\Delta \mathbf{x}^{k+1}$
 into \eqref{x_optimality_rewrite} and using the dual update relation:
\begin{align*}
\mathbf{0} &= \mathbf{v}^k - A^\mathrm{T} {\boldsymbol{\lambda}}^k + \rho^k A^\mathrm{T} (\mathbf{r}^{k+1} - A\Delta \mathbf{x}^{k+1}) + Q^k \Delta \mathbf{x}^{k+1} \\
&= \mathbf{v}^k - (A^\mathrm{T} {\boldsymbol{\lambda}}^{k+1} + \rho^k A^\mathrm{T} \mathbf{r}^{k+1}) + \rho^k A^\mathrm{T} \mathbf{r}^{k+1} \\&\quad + (Q^k - \rho^k A^\mathrm{T} A) \Delta \mathbf{x}^{k+1} \\
&= \mathbf{v}^k - A^\mathrm{T} {\boldsymbol{\lambda}}^{k+1} + S^k \Delta \mathbf{x}^{k+1},
\end{align*}
where $S^k \triangleq Q^k - \rho^k A^\mathrm{T} A$. Thus, we obtain the following relation
\begin{align}
A^\mathrm{T} {\boldsymbol{\lambda}}^{k+1} = \mathbf{v}^k + S^k \Delta \mathbf{x}^{k+1}.
\end{align}
Considering the difference between consecutive iterations:
\[
A^\mathrm{T} \Delta \boldsymbol{\lambda}^{k+1}=S^k \Delta \mathbf{x}^{k+1} + \mathbf{v}^k - \mathbf{v}^{k-1} - S^{k-1} \Delta \mathbf{x}^k.
\]

Our goal is to bound the norm $\| \Delta \boldsymbol{\lambda}^{k+1} \|^2$. Since $\sigma_{\min}(A)=1$ denotes the smallest non-zero singular value of $A$, we have the inequality $\| A^\mathrm{T} u \|^2 \ge \sigma_{\min}^2(A) \| u \|^2=\| u \|^2$ for any vector $u$ in the range of $A$. Thus, we establish the following fundamental inequality for the dual residual:
\begin{align}\label{diff-lambda-inequality}
\|\Delta {\boldsymbol{\lambda}}^{k+1}\|^2 \le  \|A^\mathrm{T} \Delta {\boldsymbol{\lambda}}^{k+1}\|^2.
\end{align}
Applying the inequality $\|a + b + c\|^2 \le (1+\theta)\|a\|^2 + 2(1+\frac{1}{\theta})(\|b\|^2 + \|c\|^2)$ for any $\theta > 0$, we can decompose the squared norm of the right-hand side as:
\begin{align}\label{A-diff-lambda-bound}
\|A^\mathrm{T} \Delta {\boldsymbol{\lambda}}^{k+1}\|^2 \
&\le \left ( 1+\theta \right ) \|S^k \Delta \mathbf{x}^{k+1}\|^2 \notag \\&+ 2(1+\frac{1}{\theta})\left ( \|\mathbf{v}^k - \mathbf{v}^{k-1} \|^2 + \|S^{k-1} \Delta \mathbf{x}^k \|^2 \right ). 
\end{align}

Next, utilizing the smoothness of $F$ (Assumption \ref{Ass: mean-squared smoothness}), we can bound $\|\mathbf{v}^k - \mathbf{v}^{k-1}\|^2$ as follows:
\begin{align}\label{diff-v-bound}
&\|\mathbf{v}^k - \mathbf{v}^{k-1}\|^2 \\
&= 2\|\mathcal{E}^k - \mathcal{E}^{k-1} +\nabla F(\mathbf{x}^k) -\nabla F(\mathbf{x}^{k-1})\|^2 \notag\\
&\le 2\|\mathcal{E}^k - \mathcal{E}^{k-1}\|^2 +2\|\nabla F(\mathbf{x}^k) -\nabla F(\mathbf{x}^{k-1})\|^2 \notag\\
&\le 2\|\mathcal{E}^k - \mathcal{E}^{k-1}\|^2 +2 L^2 \|\Delta \mathbf{x}^k\|^2 \notag\\
&\le 4 (\| \mathcal{E}^k \|^2 + \| \mathcal{E}^{k-1} \|^2) +2 L^2 \|\Delta \mathbf{x}^k\|^2.
\end{align}

Finally, substituting \eqref{diff-v-bound} back into \eqref{A-diff-lambda-bound}, we obtain the upper bound stated in Lemma \ref{Upper bound of dual variable error}:
\begin{align*}
\|\Delta {\boldsymbol{\lambda}}^{k+1}\|^2 &\le  \|A^\mathrm{T} \Delta {\boldsymbol{\lambda}}^{k+1}\|^2 \\ 
&\le   (1+\theta)(\Delta \mathbf{x}^{k+1})^\mathrm{T} (S^k)^\mathrm{T} S^k \Delta \mathbf{x}^{k+1}\\
&\quad+ \left ( 2(1+\frac{1}{\theta})\|S^{k-1}\|_2^2 + 4L^2(1+\frac{1}{\theta}) \right ) \|\Delta \mathbf{x}^k\|^2 \\
&\quad+ 8(1+\frac{1}{\theta})\|\mathcal{E}^k \|^2+8(1+\frac{1}{\theta})\| \mathcal{E}^{k-1}\|^2. 
\end{align*}
The proof is complete.
$\hfill \blacksquare$

\subsection{Proof of Lemma \ref{Recursive Bound on Gradient Estimation Error}}\label{appendix: Recursive Bound on Gradient Estimation Error}
Recalling the definition of the gradient tracking error $\mathcal{E}^{k+1} \triangleq \mathbf{v}^{k+1} - \nabla F(\mathbf{x}^{k+1})$ and the update rule of $\mathbf{v}^{k+1}$ in \eqref{update v compact}, we can expand $\mathcal{E}^{k+1}$ as
\begin{align*}
    \mathcal{E}^{k+1} &= \mathbf{v}^{k+1} - \nabla F(\mathbf{x}^{k+1}) \\
    &= \nabla F(\mathbf{x}^{k+1}, \xi^{k+1}) + (1 - a^{k+1})(\mathbf{v}^k - \nabla F(\mathbf{x}^k, \xi^{k+1}))\\
    &\quad - \nabla F(\mathbf{x}^{k+1}).
\end{align*}
By adding and subtracting $(1 - a^{k+1})\nabla F(\mathbf{x}^k)$ on the right-hand side, we obtain
\begin{align*}
    \mathcal{E}^{k+1} &= (1 - a^{k+1})(\mathbf{v}^k - \nabla F(\mathbf{x}^k)) \\
    &\quad + a^{k+1} ( \nabla F(\mathbf{x}^{k+1}, \xi^{k+1}) - \nabla F(\mathbf{x}^{k+1}) ) \\
    &\quad + (1 - a^{k+1}) \bigl[ \nabla F(\mathbf{x}^{k+1}, \xi^{k+1}) - \nabla F(\mathbf{x}^k, \xi^{k+1}) \\
    &\quad - (\nabla F(\mathbf{x}^{k+1}) - \nabla F(\mathbf{x}^k)) \bigr].
\end{align*}
For brevity, we denote the three terms on the right-hand side as $T_1$, $T_2$, and $T_3$, respectively:
\begin{align*}
T_1 &= (1 - a^{k+1})(\mathbf{v}^k - \nabla F(\mathbf{x}^k)),\\
T_2 &= a^{k+1} ( \nabla F(\mathbf{x}^{k+1}, \xi^{k+1}) - \nabla F(\mathbf{x}^{k+1}) ),\\
T_3 &= (1 - a^{k+1}) \bigl[ \nabla F(\mathbf{x}^{k+1}, \xi^{k+1}) - \nabla F(\mathbf{x}^k, \xi^{k+1}) \\
&\quad - (\nabla F(\mathbf{x}^{k+1}) - \nabla F(\mathbf{x}^k)) \bigr].
\end{align*}
Let $\mathcal{F}^k \triangleq \sigma(\xi^0, \dots, \xi^k)$ be the $\sigma$-algebra generated by the random samples up to iteration $k$. Taking the conditional expectation of the squared norm of $\mathcal{E}^{k+1}$ given $\mathcal{F}^k$, we have:
\begin{align*}
\mathbb{E}[\|\mathcal{E}^{k+1}\|^2 | \mathcal{F}^k] &= \|T_1\|^2 + \mathbb{E}[\|T_2 + T_3\|^2 | \mathcal{F}^k] \\
&\quad + 2\mathbb{E}[\langle T_1, T_2 + T_3 \rangle | \mathcal{F}^k].    
\end{align*}
Note that $T_1$ is deterministic conditioned on $\mathcal{F}^k$. Furthermore, since the sample $\xi^{k+1}$ is independent of $\mathcal{F}^k$ and the stochastic gradient is unbiased (Assumption \ref{Ass: unbiasedness and variance boundedness}), we have:
\[\mathbb{E}[T_2 | \mathcal{F}^k] = 0, \ \mathbb{E}[T_3 | \mathcal{F}^k] = 0. \]
Consequently, the cross-term vanishes:
\[
\mathbb{E}[\langle T_1, T_2 + T_3 \rangle | \mathcal{F}^k] = \langle T_1, \mathbb{E}[T_2 + T_3 | \mathcal{F}^k] \rangle = 0.
\]
Therefore, we have
\begin{align*}
    &\mathbb{E}[\|\mathcal{E}^{k+1}\|^2 | \mathcal{F}^k] \\&= \|T_1\|^2 + \mathbb{E}[\|T_2 + T_3\|^2 | \mathcal{F}^k] \\
    &\le \|T_1\|^2 + 2\mathbb{E}[\|T_2\|^2 | \mathcal{F}^k] + 2\mathbb{E}[\|T_3\|^2 | \mathcal{F}^k]\\
    &\le(1 - a^{k+1})^2\|\mathcal{E}^k\|^2 + 2(a^{k+1})^2 \sigma^2\\
    &\quad + (1 - a^{k+1})^2\mathbb{E}[\| \nabla F(\mathbf{x}^{k+1}, \xi^{k+1}) - \nabla F(\mathbf{x}^k, \xi^{k+1})\|^2 | \mathcal{F}^k]\\
    &\le(1 - a^{k+1})^2\|\mathcal{E}^k\|^2 + 2(a^{k+1})^2 \sigma^2\\
    &\quad +2L^2 (1 - a^{k+1})^2\|\Delta \mathbf{x}^{k+1}\|^2,
\end{align*}
where the first inequality uses $\|a+b\|^2 \le 2\|a\|^2 + 2\|b\|^2$, the second inequality follows from the Assumption \ref{Ass: unbiasedness and variance boundedness} and $\mathbb{E}[\|X - \mathbb{E}(X)\|^2] \le \mathbb{E}[\|X\|^2]$, the last inequality follows from the Assumption \ref{Ass: mean-squared smoothness}. Taking the total expectation on both sides completes the proof. $\hfill \blacksquare$

\subsection{Proof of Lemma \ref{Single-Step Lyapunov Descent Lemma}}\label{appendix: Single-Step Lyapunov Descent Lemma}
We analyze the difference of the Lyapunov function $\Phi^{k+1} - \Phi^k$. Recall the definition of $\Phi^k$ in \eqref{eq: Lyapunov function}. We decompose its difference into three parts: the change in the augmented Lagrangian function, the change in the gradient estimation error, and the change in the primal momentum term
\begin{align}\label{Single-Step Lyapunov Descent equality}
&\Phi^{k+1} - \Phi^k \notag\\
&= \mathcal{L}_{\rho^k}(\mathbf{x}^{k+1}, \mathbf{y}^{k+1}, {\boldsymbol{\lambda}}^{k+1}) - \mathcal{L}_{\rho^{k-1}}(\mathbf{x}^k, \mathbf{y}^k, {\boldsymbol{\lambda}}^k) \notag\\
&\quad +  \frac{1}{\Gamma^{k+1}}\|\mathcal{E}^{k+1}\|^2 + \frac{C_{err}}{\rho^k }\|\mathcal{E}^{k}\|^2- \frac{1}{\Gamma^k}\|\mathcal{E}^k\|^2\notag\\
&\quad -\frac{C_{err}}{\rho^{k-1} } \|\mathcal{E}^{k-1}\|^2+ \frac{\beta_{k+1}}{2}\|\Delta \mathbf{x}^{k+1}\|^2 - \frac{\beta^k}{2}\|\Delta \mathbf{x}^k\|^2.
\end{align}
We further decompose the augmented Lagrangian function difference based on the updates of the variable sequence $(\mathbf{x}, \mathbf{y}, \boldsymbol{\lambda})$ and the parameter $\rho$
\begin{align*}
&\mathcal{L}_{\rho^k}(\mathbf{x}^{k+1}, \mathbf{y}^{k+1}, {\boldsymbol{\lambda}}^{k+1}) - \mathcal{L}_{\rho^{k-1}}(\mathbf{x}^k, \mathbf{y}^k, {\boldsymbol{\lambda}}^k) \\
&=  \underbrace{\mathcal{L}_{\rho^k}(\mathbf{x}^k, \mathbf{y}^k, {\boldsymbol{\lambda}}^k) - \mathcal{L}_{\rho^{k-1}}(\mathbf{x}^k, \mathbf{y}^k, {\boldsymbol{\lambda}}^k)}_{ \text{ (a) Parameter } \rho \text{ Change} } \\
&\quad + \underbrace{\mathcal{L}_{\rho^k}(\mathbf{x}^k, \mathbf{y}^{k+1}, {\boldsymbol{\lambda}}^k) - \mathcal{L}_{\rho^k}(\mathbf{x}^k, \mathbf{y}^k, {\boldsymbol{\lambda}}^k)}_{\text{(b) Primal } \mathbf{y} \text{ Update}} \\
&\quad + \underbrace{\mathcal{L}_{\rho^k}(\mathbf{x}^{k+1}, \mathbf{y}^{k+1}, {\boldsymbol{\lambda}}^k) - \mathcal{L}_{\rho^k}(\mathbf{x}^k, \mathbf{y}^{k+1}, {\boldsymbol{\lambda}}^k)}_{\text{(c) Primal } \mathbf{x} \text{ Update}} \\
&\quad + \underbrace{\mathcal{L}_{\rho^k}(\mathbf{x}^{k+1}, \mathbf{y}^{k+1}, {\boldsymbol{\lambda}}^{k+1}) - \mathcal{L}_{\rho^k}(\mathbf{x}^{k+1}, \mathbf{y}^{k+1}, {\boldsymbol{\lambda}}^k)}_{\text{(d) Dual } \boldsymbol{\lambda} \text{ Update}}.
\end{align*}
Next, we will analyze the upper bounds of terms (a), (b), (c), and (d) separately.

Analysis of (a): Since the penalty parameter is non-decreasing ($\rho^k \ge \rho^{k-1}$), the change yields:
\begin{align*}
(\text{a}) =\frac{\rho^k - \rho^{k-1}}{2} \|\mathbf{r}^k\|^2 . 
\end{align*}

Analysis of (b):
Because $\mathbf{y}^{k+1}$ is the minimum point of $\mathcal{L}_{\rho^k}(\mathbf{x}^k, \mathbf{y}, {\boldsymbol{\lambda}}^k)$, we have:
\begin{align*}
\text{(b)}=\mathcal{L}_{\rho^k}(\mathbf{x}^k, \mathbf{y}^{k+1}, {\boldsymbol{\lambda}}^k) - \mathcal{L}_{\rho^k}(\mathbf{x}^k, \mathbf{y}^k, {\boldsymbol{\lambda}}^k) \le 0.
\end{align*}

Analysis of (c): The change with respect to $\mathbf{x}$ can be written as
\begin{align}\label{(c) expression_v1}
\text{(c)}&= F(\mathbf{x}^{k+1})- F(\mathbf{x}^{k})-\langle {\boldsymbol{\lambda}}^k, A\Delta \mathbf{x}^{k+1} \rangle \notag\\
&\quad+\frac{\rho^k}{2} ( \|\mathbf{r}^{k+1}\|^2 - \|\tilde{\mathbf{r}}^k\|^2 ) .
\end{align}
Using the $L$-smoothness of $F$ (implied by Assumption \ref{Ass: mean-squared smoothness}), we bound the function difference:
\begin{align}\label{F smooth}
F(\mathbf{x}^{k+1}) - F(\mathbf{x}^k) \le \langle \nabla F(\mathbf{x}^k), \Delta \mathbf{x}^{k+1} \rangle + \frac{L}{2} \|\Delta \mathbf{x}^{k+1}\|^2    .
\end{align}
According to the first-order optimality condition \eqref{x_optimality}, we have
\[\mathbf{v}^k - A^\mathrm{T} {\boldsymbol{\lambda}}^k + \rho^k A^\mathrm{T} \tilde{\mathbf{r}}^k + Q^k \Delta \mathbf{x}^{k+1} = 0,\]
therefore, the inner product term becomes:
\begin{align}\label{diff F diif x}
&\langle \nabla F(\mathbf{x}^k), \Delta \mathbf{x}^{k+1} \rangle \notag\\ &= \langle \mathbf{v}^k - \mathcal{E}^k, \Delta \mathbf{x}^{k+1} \rangle \notag\\
&=\langle A^\mathrm{T} {\boldsymbol{\lambda}}^k - \rho^k A^\mathrm{T} \tilde{\mathbf{r}}^k - Q^k \Delta \mathbf{x}^{k+1} - \mathcal{E}^k, \Delta \mathbf{x}^{k+1} \rangle .
\end{align}
Substituting (\ref{F smooth}) and (\ref{diff F diif x}) back into (\ref{(c) expression_v1}):
\begin{align}\label{(c) expression_v2}
\text{(c)}&\le \langle \nabla F(\mathbf{x}^k), \Delta \mathbf{x}^{k+1} \rangle + \frac{L}{2} \|\Delta \mathbf{x}^{k+1}\|^2 -\langle {\boldsymbol{\lambda}}^k, A\Delta \mathbf{x}^{k+1} \rangle \notag\\
&\quad  +\frac{\rho^k}{2} ( \|\mathbf{r}^{k+1}\|^2 - \|\tilde{\mathbf{r}}^k\|^2 )\notag\\
&\le \langle - \rho^k A^\mathrm{T} \tilde{\mathbf{r}}^k - Q^k \Delta \mathbf{x}^{k+1} - \mathcal{E}^k, \Delta \mathbf{x}^{k+1} \rangle + \frac{L}{2} \|\Delta \mathbf{x}^{k+1}\|^2\notag\\
&\quad  +\frac{\rho^k}{2} ( \|\mathbf{r}^{k+1}\|^2 - \|\tilde{\mathbf{r}}^k\|^2 ).
\end{align}
Utilizing equation
\begin{align*}
\|\mathbf{r}^{k+1}\|^2 &= \|\tilde{\mathbf{r}}^k + A \Delta \mathbf{x}^{k+1}\|^2 \\
&= \|\tilde{\mathbf{r}}^k\|^2 + 2 \langle \tilde{\mathbf{r}}^k, A \Delta \mathbf{x}^{k+1} \rangle + \|A \Delta \mathbf{x}^{k+1}\|^2,
\end{align*}
and
\begin{align*}
&-\rho^k \langle A^\mathrm{T} \tilde{\mathbf{r}}^k, \Delta \mathbf{x}^{k+1} \rangle + \frac{\rho^k}{2} ( \|\mathbf{r}^{k+1}\|^2 - \|\tilde{\mathbf{r}}^k\|^2 )\\
&= -\rho^k \langle \tilde{\mathbf{r}}^k, A\Delta \mathbf{x}^{k+1} \rangle + \frac{\rho^k}{2}(2 \langle \tilde{\mathbf{r}}^k, A \Delta \mathbf{x}^{k+1} \rangle + \|A \Delta \mathbf{x}^{k+1}\|^2)\\
&= \frac{\rho^k}{2} \|A\Delta \mathbf{x}^{k+1}\|^2,
\end{align*}
we can further simplify (\ref{(c) expression_v2}) as:
\begin{align}
\text{(c)}\label{(c) expression_v3}
\le&  - (\Delta \mathbf{x}^{k+1})^\mathrm{T} \left( Q^k - \frac{\rho^k}{2}A^\mathrm{T}A - \frac{L}{2}I \right) (\Delta \mathbf{x}^{k+1}) \notag\\
&-\langle \mathcal{E}^k, \Delta \mathbf{x}^{k+1} \rangle.
\end{align}
Applying Young's inequality to the cross term $\langle \mathcal{E}^k, \Delta x^{k+1} \rangle$, we yield
\[-\langle \mathcal{E}^k, \Delta x^{k+1} \rangle \le \frac{1}{2\mu^k}\|\mathcal{E}^k\|^2 + \frac{\mu^k}{2}\|\Delta x^{k+1}\|^2,\]
we obtain the final bound for (c):
\begin{align}
\text{(c)}\label{(c) expression_v4}
\le& - (\Delta \mathbf{x}^{k+1})^\mathrm{T} \left( Q^k - \frac{\rho^k}{2}A^\mathrm{T}A - \frac{L+\mu^k}{2}I \right) (\Delta \mathbf{x}^{k+1}) \notag\\
&+\frac{1}{2\mu^k}\|\mathcal{E}^k\|^2.
\end{align}

Analysis of (d): Using the dual update rule $\boldsymbol{\lambda}^{k+1} ={\boldsymbol{\lambda}}^k -\rho^k\mathbf{r}^{k+1}$, we get
\begin{align}\label{(d) expression}
\text{(d)}&= -\langle {\boldsymbol{\lambda}}^{k+1} - {\boldsymbol{\lambda}}^k, \mathbf{r}^{k+1} \rangle \notag\\
&=\rho^k\|\mathbf{r}^{k+1}\|^2 \notag\\
&=-\frac{\rho^k}{2}\|\mathbf{r}^{k+1}\|^2+\frac{3}{2}\rho^k\|\mathbf{r}^{k+1}\|^2 \notag\\
&=-\frac{\rho^k}{2}\|\mathbf{r}^{k+1}\|^2+\frac{3}{2\rho^k}\|\Delta {\boldsymbol{\lambda}}^{k+1}\|^2 .
\end{align}
Combining (a), (b), (c), (d) and utilizing the upper bound of $\|\Delta {\boldsymbol{\lambda}}^{k+1}\|^2$ (Lemma \ref{Upper bound of dual variable error}) we obtain
\begin{align}\label{part I}
&\mathcal{L}_{\rho^k}(\mathbf{x}^{k+1}, \mathbf{y}^{k+1}, {\boldsymbol{\lambda}}^{k+1}) - \mathcal{L}_{\rho^{k-1}}(\mathbf{x}^k, \mathbf{y}^k, {\boldsymbol{\lambda}}^k) \notag\\
\le & - (\Delta \mathbf{x}^{k+1})^\mathrm{T} \Bigl[ Q^k - \frac{L}{2}I -\frac{\mu^k}{2}I- \frac{\rho^k}{2}A^\mathrm{T}A \notag\\
&- \frac{3(1+\theta)}{2\rho^k}(S^k)^\mathrm{T}S^k \Bigr] (\Delta \mathbf{x}^{k+1})  + \frac{3D^k}{2\rho^k } \|\Delta \mathbf{x}^k\|^2 \notag\\
& + \left( \frac{1}{2\mu^k} + \frac{12(1+\frac{1}{\theta})}{\rho^k } \right) \|\mathcal{E}^k\|^2 + \frac{12(1+\frac{1}{\theta})}{\rho^k } \|\mathcal{E}^{k-1}\|^2 \notag\\
& + \frac{\rho^k - \rho^{k-1}}{2} \|\mathbf{r}^k\|^2 - \frac{\rho^k}{2} \|\mathbf{r}^{k+1}\|^2,
\end{align}
where 
\[D^k = 2(1+\frac{1}{\theta})\|S^{k-1}\|^2 + 4L^2(1+\frac{1}{\theta}).\]
Simultaneously, according to Lemma \ref{Recursive Bound on Gradient Estimation Error}, we have
\begin{align}\label{part II}
& \frac{1}{\Gamma^{k+1}}\|\mathcal{E}^{k+1}\|^2 - \frac{1}{\Gamma^k}\|\mathcal{E}^k\|^2 +C_{err}\left ( \frac{\|\mathcal{E}^{k}\|^2}{\rho^k} -\frac{\|\mathcal{E}^{k-1}\|^2}{\rho^{k-1}} \right ) 
\notag\\
&\le \frac{2L^2(1-a^{k+1})^2}{\Gamma^{k+1}} \|\Delta \mathbf{x}^{k+1}\|^2 \notag\\
& + \left( \frac{(1-a^{k+1})^2}{\Gamma^{k+1}} - \frac{1}{\Gamma^k} + \frac{C_{err}}{\rho^k} \right) \|\mathcal{E}^k\|^2 \notag\\
& - \frac{C_{err}}{\rho^{k-1}} \|\mathcal{E}^{k-1}\|^2 + \frac{2(a^{k+1})^2 \sigma^2}{\Gamma^{k+1}}.
\end{align}
Substituting formulas (\ref{part I}) and (\ref{part II}) into formula (\ref{Single-Step Lyapunov Descent equality}), we yield the comprehensive single-step descent inequality:
\begin{align}\label{Single-Step descent complete expression}
&\Phi^{k+1} - \Phi^k\notag\\
\le & - (\Delta \mathbf{x}^{k+1})^\mathrm{T} C_{\mathbf{x}}^k (\Delta \mathbf{x}^{k+1})  - C_{\mathcal{E}}^k \|\mathcal{E}^k\|^2 \notag\\
& + \left(\frac{3D^k}{2\rho^k } - \frac{\beta^k}{2}\right) \|\Delta \mathbf{x}^k\|^2  + \left(\frac{12(1+\frac{1}{\theta})}{\rho^k } - \frac{C_{err}}{\rho^{k-1}}\right) \|\mathcal{E}^{k-1}\|^2 \notag\\
&  - \frac{\rho^k}{2} \|\mathbf{r}^{k+1}\|^2 + \frac{\rho^k - \rho^{k-1}}{2} \|\mathbf{r}^k\|^2 + \frac{2(a^{k+1})^2 \sigma^2}{\Gamma^{k+1}},
\end{align}
where
\begin{align*}
C_{\mathbf{x}}^k &= \left( Q^k - \frac{\rho^k}{2} A^\mathrm{T} A \right) - \frac{3(1+\theta)}{2\rho^k } (S^k)^\mathrm{T} S^k \\
&\quad - \left( \frac{\mu^k}{2} + \frac{\beta_{k+1}}{2} + \frac{L}{2} + \frac{2L^2(1-a^{k+1})^2}{\Gamma^{k+1}} \right) I,\\
C_{\mathcal{E}}^k &=\frac{1}{\Gamma^k} - \frac{(1-a^{k+1})^2}{\Gamma^{k+1}} -  \frac{1}{2\mu^k} - \frac{12(1+\frac{1}{\theta})}{\rho^k } - \frac{C_{err}}{\rho^k}. 
 \end{align*}
Let parameters $\beta^k$ and $C_{err}$ satisfy conditions \eqref{beta^k} and \eqref{C_err} respectively, we have
\begin{align*}
\frac{3D^k}{2\rho^k } - \frac{\beta^k}{2}\le0,\ 
\frac{12(1+\frac{1}{\theta})}{\rho^k } - \frac{C_{err}}{\rho^{k-1}}\le0.
\end{align*}
So we can remove the $\|\Delta \mathbf{x}^k\|^2$ and $\|\mathcal{E}^{k-1}\|^2$ related terms in \eqref{Single-Step descent complete expression} to obtain \eqref{Single-Step Lyapunov Descent}.
$\hfill \blacksquare$
\subsection{Proof of Theorem \ref{potential function accumulation}}\label{appendix: potential function accumulation}
Based on the derivation process of the Lemma \ref{Single-Step Lyapunov Descent Lemma}, we establish the following lower bounds for the parameter sequences:
\begin{align*}
C_{\mathbf{x}}^k \ge k^{1/3} C_\mathbf{x},\\ C_{\mathcal{E}}^k \ge C_{\mathcal{E}} k^{-\frac{1}{3}},
\end{align*}
where 
\begin{align*}
C_\mathbf{x} &= \left( C_\eta - \frac{c_\rho}{2} A^\mathrm{T} A \right)\\
&\quad- \frac{3(1+\theta)}{2c_\rho } (C_\eta - c_\rho A^\mathrm{T} A)^\mathrm{T} (C_\eta - c_\rho A^\mathrm{T} A)\\
&\quad - \left( \frac{c_\mu}{2} + \frac{c_\beta}{2} +\frac{L}{2} + 2L^2 c_\gamma \right) I,\\
C_{\mathcal{E}} &= 2c_a c_\gamma  - \frac{1}{2c_\mu} - \frac{12(1+\frac{1}{\theta})}{c_\rho } - \frac{C_{err}}{c_\rho}.
\end{align*}

Summing the recursive inequality (\ref{Single-Step Lyapunov Descent}) derived in Lemma \ref{Single-Step Lyapunov Descent Lemma} from $k=1$ to $K$, we obtain
\begin{align}\label{eq:sum_lyapunov}
&\sum_{k=1}^K(\Phi^{k+1} - \Phi^k )\notag\\
&\le  - \sum_{k=1}^K(\Delta \mathbf{x}^{k+1})^\mathrm{T} C_{\mathbf{x}}^k (\Delta \mathbf{x}^{k+1})  - \sum_{k=1}^KC_{\mathcal{E}}^k \|\mathcal{E}^k\|^2  \notag\\
& + \sum_{k=1}^K\left ( \frac{\rho^k - \rho^{k-1}}{2} \|\mathbf{r}^k\|^2 -\frac{\rho^k }{2} \|\mathbf{r}^{k+1}\|^2\right )  + \sum_{k=1}^K\frac{2(a^{k+1})^2 \sigma^2}{\Gamma^{k+1}}.
\end{align}
For sufficiently large $c_\eta$, $C_\mathbf{x}$ is positive definite, and its minimum eigenvalue is positive. Using the Rayleigh quotient inequality, we have 
\[ (\Delta \mathbf{x}^{k+1})^\mathrm{T} C_\mathbf{x} (\Delta \mathbf{x}^{k+1}) \ge \lambda_{\min}(C_\mathbf{x}) \|\Delta \mathbf{x}^{k+1}\|^2. \]
This directly implies
\[- (\Delta \mathbf{x}^{k+1})^\mathrm{T} C_{\mathbf{x}}^k (\Delta \mathbf{x}^{k+1}) \le - \lambda_{\min}(C_\mathbf{x}) k^{1/3} \|\Delta \mathbf{x}^{k+1}\|^2.\]

Regarding the residual terms, by telescoping the sum, we can rearrange it as follows:
\begin{align*}
&\sum_{k=1}^K\left ( \frac{\rho^k - \rho^{k-1}}{2} \|\mathbf{r}^k\|^2 -\frac{\rho^k }{2} \|\mathbf{r}^{k+1}\|^2\right )\\
&=\frac{\rho^1 -\rho^0}{2}\|\mathbf{r}^1\|^2 +\sum_{k=1}^{K-1}\left ( \frac{\rho^{k+1} - \rho^k}{2}- \frac{ \rho^k}{2}\right )\|\mathbf{r}^{k+1}\|^2 \\
&\quad -\frac{\rho^K }{2} \|\mathbf{r}^{K+1}\|^2\\
&\le \frac{\rho^1 -\rho^0}{2}\|\mathbf{r}^1\|^2 + \sum_{k=1}^{K}\left ( \frac{\rho^{k+1} - \rho^k}{2}- \frac{ \rho^k}{2}\right )\|\mathbf{r}^{k+1}\|^2.
\end{align*}
Substituting the parameter scaling $\rho^k = c_\rho k^{1/3}$, the coefficient inside the summation can be strictly bounded by using the inequality $(k+1)^{\frac{1}{3}} \le k^{\frac{1}{3}} + \frac{1}{3}k^{-\frac{2}{3}}$

\begin{align*}
\frac{\rho^{k+1} - \rho^k}{2}-\frac{\rho^k}{2}&=\frac{c_\rho(k+1)^{\frac{1}{3}}-c_\rho k^{\frac{1}{3}}}{2} -\frac{c_\rho}{2}k^{\frac{1}{3}} \\
&\le\frac{c_\rho (k^{\frac{1}{3}}+\frac{1}{3}k^{-\frac{2}{3}})-c_\rho k^{\frac{1}{3}}}{2} -\frac{c_\rho}{2}k^{\frac{1}{3}}\\
&\le \frac{c_\rho}{6}k^{-\frac{2}{3}}-\frac{c_\rho}{2}k^{\frac{1}{3}}\\
&\le-\frac{c_\rho}{4}k^{\frac{1}{3}}.
\end{align*}

As for the noise term, the sum is logarithmically bounded:
\[\sum_{k=1}^K\frac{2(a^{k+1})^2 \sigma^2}{\Gamma^{k+1}} \le 2\sigma^2c_a^2c_\gamma\left (1+ \ln{K} \right ).\]

Substituting all the derived bounds back into \eqref{eq:sum_lyapunov}, we obtain
\begin{align*}
&\Phi^{K+1} - \Phi^1 \\\le & - \lambda_{\min}(C_\mathbf{x}) k^{1/3} \|\Delta \mathbf{x}^{k+1}\|^2  - \sum_{k=1}^KC_{\mathcal{E}} k^{-\frac{1}{3}} \|\mathcal{E}^k\|^2  \\
-&\sum_{k=1}^K\frac{c_\rho}{4} k^{\frac{1}{3}}\|\mathbf{r}^{k+1}\|^2 + \frac{\rho^1 -\rho^0}{2}\|\mathbf{r}^1\|^2 + 2\sigma^2c_a^2c_\gamma\left (1+ \ln{K} \right ).
\end{align*}
It follows from Assumption \ref{Ass: f,g lower bounded} that there exists a lower bound $\Phi^{\star}$ for the sequence $\{\Phi^k\}$. By moving the negative summation terms to the left-hand side and replacing $\Phi^{K+1}$ with $\Phi^{\star}$, we deduce
\begin{align*}
&\sum_{k=1}^K  \bigg( C_{\mathcal{E}} k^{-1/3} \|\mathcal{E}^k\|^2 + \lambda_{\min}(C_\mathbf{x}) k^{1/3} \|\Delta \mathbf{x}^{k+1}\|^2 \\
&\qquad + \frac{c_\rho}{4} k^{1/3} \|\mathbf{r}^{k+1}\|^2  \bigg )\\
&\le \Phi^1 -\Phi^\star + \frac{\rho^1 -\rho^0}{2}\|\mathbf{r}^1\|^2+ 2\sigma^2c_a^2c_\gamma\left (1+ \ln{K} \right ),
\end{align*}
Dividing both sides by the minimum coefficient, we finally establish the accumulation bound \eqref{eq:accumulation bound}.
$\hfill \blacksquare$
\subsection{Proof of Theorem \ref{Convergence rate theorem}}\label{appendix: Convergence rate theorem}
Based on the definition of an $\epsilon$-stationary point (Definition \ref{def:epsilon_stationary_point}) and the subgradient of the Lagrangian function \eqref{subdifferential_of_Lagrangian}, the squared distance can be deposed as:
\begin{align*}
&\mathbb{E}\left[ \text{dist}^2\left( 0,\partial \mathcal{L}\left( \mathbf {x},\mathbf {y},\bm{\lambda}  \right) \right) \right] \\
&= \mathbb{E}\left[ \| \nabla F\left( \mathbf {x}\right)  -A^{\mathrm{T}}\bm{\lambda}  \| ^2\right] + \mathbb{E}\left[\text{dist}^2 \left( \partial H\left( \mathbf {x} \right) , B^{\mathrm{T}}\bm{\lambda} \right)  \right]\\
&\quad + \mathbb{E}\left[ \lVert A\mathbf {x}+B\mathbf {y} \rVert ^2 \right],
\end{align*}

We proceed to bound each of these three terms separately.\\
\textbf{1. Bound on gradient stationarity:} From the first-order optimality condition  (\ref{x_optimality_rewrite}), we have
\begin{align*}
\nabla F(\mathbf{x}^k) - A^\mathrm{T} {\boldsymbol{\lambda}}^k = -\mathcal{E}^k -\rho^k A^\mathrm{T} \mathbf{r}^{k+1} - S^k \Delta \mathbf{x}^{k+1}.
\end{align*}
Applying the basic inequality $\|a+b+c\|^2 \le 3\|a\|^2 + 3\|b\|^2 + 3\|c\|^2$, we bound the squared norm:
\begin{align*}
&\|\nabla F(\mathbf{x}^k) - A^\mathrm{T} {\boldsymbol{\lambda}}^k\|^2 \\
&\le 3 \|\mathcal{E}^k\|^2 +  3 (\rho^k)^2 \|A^\mathrm{T}\|^2 \|\mathbf{r}^{k+1}\|^2+ 3 \|S^k\|_2^2 \|\Delta \mathbf{x}^{k+1}\|^2,
\end{align*}
Substituting the explicit parameter settings $\rho^k = c_\rho k^{1/3}$ and $S^k = Q^k - \rho^k A^\mathrm{T} A= k^{1/3} \left( C_\eta - c_\rho A^\mathrm{T} A \right)$, we obtain the explicit bounds for the coefficients:
\[3 (\rho^k)^2 \|A^\mathrm{T}\|^2 =3\lambda_{\max}(A^\mathrm{T} A)c_\rho^2 k^{\frac{2}{3}},\]
\begin{align*}
3\|S^k\|_2^2 &=3 k^{2/3} \left\| C_\eta - c_\rho A^\mathrm{T} A \right\|_2^2.
\end{align*}
Thus,
\begin{align*}
&\|\nabla F(\mathbf{x}^k) - A^\mathrm{T} {\boldsymbol{\lambda}}^k\|^2 \\&\le C_5 \left( \|\mathcal{E}^k\|^2 + k^{\frac{2}{3}}\|\Delta \mathbf{x}^{k+1}\|^2+  k^{\frac{2}{3}}\|\mathbf{r}^{k+1}\|^2\right),
\end{align*}
where $C_5 = \max \left\{ 3,\ 3 \left\| C_\eta - c_\rho A^\mathrm{T} A \right\|_2^2, \ 3\lambda_{\max}(A^\mathrm{T} A)c_\rho^2  \right\}$.

Summing this inequality from $k=1$ to $K$ and utilizing the accumulation bound in Theorem \ref{potential function accumulation}, we have
\begin{align*}
&\sum_{k=1}^K\|\nabla F(\mathbf{x}^k) - A^\mathrm{T} {\boldsymbol{\lambda}}^k\|^2 \\&\le \sum_{k=1}^K C_5 k^\frac{1}{3}\left( k^{-\frac{1}{3}}\|\mathcal{E}^k\|^2 + k^\frac{1}{3} \|\Delta \mathbf{x}^{k+1}\|^2 + k^\frac{1}{3} \|\mathbf{r}^{k+1}\|^2 \right)\\
&\le C_5 \frac{\Phi^{1} - \Phi^{\star} + \frac{\rho^1 -\rho^0}{2}\|\mathbf{r}^1\|^2+ 2\sigma^2c_a^2c_\gamma\left (1+ \ln{K} \right )}{\min \left (C_{\mathcal{E}},\lambda_{\min}(C_\mathbf{x}),\frac{c_\rho}{4} \right )} K^\frac{1}{3}.
\end{align*}
To simplify the presentation, we define the constants $\mathcal{M}_1$ and $\mathcal{M}_2$ as follows:
\begin{align*}
\mathcal{M}_1 &\triangleq \frac{\Phi^{1} - \Phi^{\star} + \frac{\rho^1 -\rho^0}{2}\|\mathbf{r}^1\|^2 + 2\sigma^2c_a^2c_\gamma}{\min \left (C_{\mathcal{E}},\lambda_{\min}(C_\mathbf{x}),\frac{c_\rho}{4} \right )},\\   
\mathcal{M}_2 &\triangleq \frac{2\sigma^2c_a^2c_\gamma}{\min \left (C_{\mathcal{E}},\lambda_{\min}(C_\mathbf{x}),\frac{c_\rho}{4} \right )}.
\end{align*}
\textbf{2. Bound on subgradient stationarity:} Reviewing the update process of $\mathbf{y}$ in (\ref{y_update_compact}), the first-order optimality condition implies the existence of a subgradient $\mathbf{g} \in \partial H(\mathbf{y}^k)$ such that:
\[\mathbf{g} - B^\mathrm{T} {\boldsymbol{\lambda}}^{k-1} + \rho^{k-1} B^\mathrm{T} (A\mathbf{x}^{k-1} + B\mathbf{y}^{k}) = 0.\]
Substituting the dual update relation (\ref{update lambda compact}), we obtain
\begin{align*}
B^\mathrm{T} {\boldsymbol{\lambda}}^k - \mathbf{g} &= B^\mathrm{T} \left[ {\boldsymbol{\lambda}}^{k-1} - \rho^{k-1} (A\mathbf{x}^{k} + B\mathbf{y}^{k}) \right] \\
&\quad - B^\mathrm{T} \left[ {\boldsymbol{\lambda}}^{k-1} - \rho^{k-1} (A\mathbf{x}^{k-1} + B\mathbf{y}^{k}) \right] \\
&= -\rho^{k-1} B^\mathrm{T} A \Delta \mathbf{x}^{k}.
\end{align*}
Taking the squared norm and using $\rho^{k-1} \le \rho^k$ yield the explicit bound:
\begin{align*}
\text{dist}^2(B^\mathrm{T} {\boldsymbol{\lambda}}^k, \partial H(\mathbf{y}^k))&=\| B^\mathrm{T} {\boldsymbol{\lambda}}^k - \mathbf{g} \|^2 \\
&\le   (\rho^k)^2 \|B\| ^2\|A\|^2\|\Delta \mathbf{x}^{k}\|^2. 
\end{align*}
Therefore, the accumulation bound is
\begin{align*}
&\sum_{k=1}^K\text{dist}^2(B^\mathrm{T} {\boldsymbol{\lambda}}^k, \partial H(\mathbf{y}^k))\\
&\le \sum_{k=1}^K c_\rho^2 \|B\| ^2\|A\|^2 k^\frac{1}{3}\left(  k^\frac{1}{3} \|\Delta \mathbf{x}^{k}\|^2  \right)\\
&\le c_\rho^2 \|B\| ^2\|A\|^2(\mathcal{M}_1 + \mathcal{M}_2 \ln K) K^{1/3}.
\end{align*}
\textbf{3. Bound on residual error:} Similarly, adjusting the summation index to utilize Theorem \ref{potential function accumulation}, we have
\begin{align*}
&\sum_{k=1}^K\|A\mathbf{x}^k + B\mathbf{y}^k\|^2=\sum_{k=1}^K\|\mathbf{r}^k\|^2 \le \sum_{k=1}^K k^{\frac{1}{3}}\|\mathbf{r}^k\|^2 \\
&\le \frac{\Phi^{1} - \Phi^{\star} +\frac{\rho^1 -\rho^0}{2}\|\mathbf{r}^1\|^2+ 2\sigma^2c_a^2c_\gamma\left (1+ \ln{K} \right )}{\min \left (C_{\mathcal{E}},\lambda_{\min}(C_\mathbf{x}),\frac{c_\rho}{4} \right )}\\
&\le \mathcal{M}_1 + \mathcal{M}_2 \ln K.
\end{align*}
\textbf{4. Global Convergence Rate:}
Combining all three components and averaging over $K$ iterations, we characterize the optimal stationarity measure:
\begin{align*}
&\min_{1\le k \le K} \mathbb{E}\left[ \mathrm{dist}^2 \left(0, \partial \mathcal{L}(\mathbf{x}^k, \mathbf{y}^k, \boldsymbol{\lambda}^k)\right) \right]\\
&\le \frac{1}{K}\sum_{k=1}^K \mathbb{E}\left[ \mathrm{dist}^2 \left(0, \partial \mathcal{L}(\mathbf{x}^k, \mathbf{y}^k, \boldsymbol{\lambda}^k)\right) \right] \\
&= \frac{1}{K}\sum_{k=1}^K \bigg\{ \mathbb{E}\left\| \nabla F(\mathbf{x}^k) - A^\mathrm{T} \boldsymbol{\lambda}^k \right\|^2\\
 &\qquad + \mathbb{E}\left[ \mathrm{dist}^2\left(B^\mathrm{T} \boldsymbol{\lambda}^k, \partial H(\mathbf{y}^k)\right) \right] 
 + \mathbb{E}\left\| A\mathbf{x}^k + B\mathbf{y}^k \right\|^2 \bigg\}\\
&=S_1 K^{-\frac{2}{3}}+S_2 K^{-\frac{2}{3}} +S_3 K^{-1},
\end{align*}
where the dominant term coefficients are
\begin{align*}
S_1 &=C_5 (\mathcal{M}_1 + \mathcal{M}_2 \ln K),\\
S_2 &= c_\rho^2 \|B\| ^2\|A\|^2(\mathcal{M}_1 + \mathcal{M}_2 \ln K), \\
S_3 &=  \mathcal{M}_1 + \mathcal{M}_2 \ln K .
\end{align*}
This establishes the optimal convergence rate of $\mathcal{O}(K^{-2/3})$ (or equivalently $\mathcal{O}(\epsilon^{-1.5})$ oracle complexity).
$\hfill \blacksquare$


\bibliographystyle{IEEEtran}
\bibliography{reference_zym.bib}

\begin{IEEEbiography}
[{\includegraphics[width=1in,height=1.25in,clip,keepaspectratio]{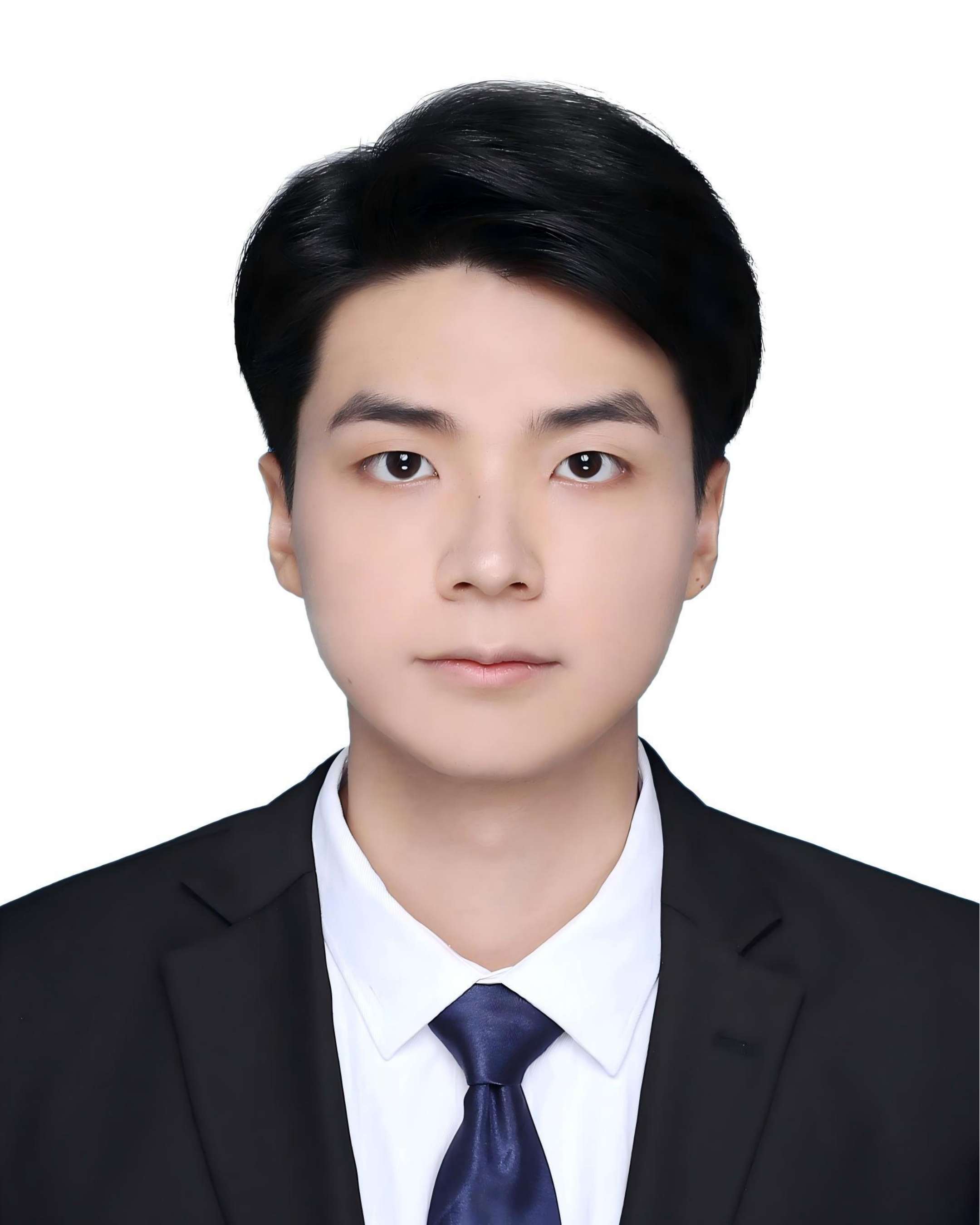}}]{Yangming Zhang}
received the B.S. degree in underwater acoustic engineering and the M.E. degree in control science and engineering from Northwestern Polytechnical University, Xi'an, China, in 2020 and 2023, respectively. He is currently working toward the Ph.D degree in Sun Yat-sen University. His current research interests include distributed optimization and control, machine learning, and multi-robot collaboration.
\end{IEEEbiography}

\begin{IEEEbiography}
[{\includegraphics[width=1in,height=1.25in,clip,keepaspectratio]{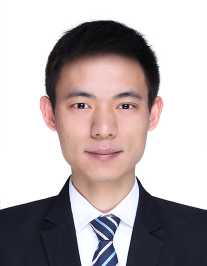}}]{Yongyang Xiong}
received the B.S. degree in information and computational science, the M.E. and Ph.D. degrees in control science and engineering from Harbin Institute of Technology, Harbin, China, in 2012, 2014, and 2020, respectively. From 2017 to 2018, he was a joint Ph.D. student with the School of Electrical and Electronic Engineering, Nanyang Technological University, Singapore. From 2021 to 2024, he was a Postdoctoral Fellow with the Department of Automation, Tsinghua University, Beijing, China. Currently, he is an associate professor with the School of Intelligent Systems Engineering, Sun Yat-sen University, Shenzhen, China. His current research interests include networked control systems, distributed optimization and learning, multi-agent reinforcement learning and their applications.
\end{IEEEbiography}

\begin{IEEEbiography}
[{\includegraphics[width=1in,height=1.25in,clip,keepaspectratio]{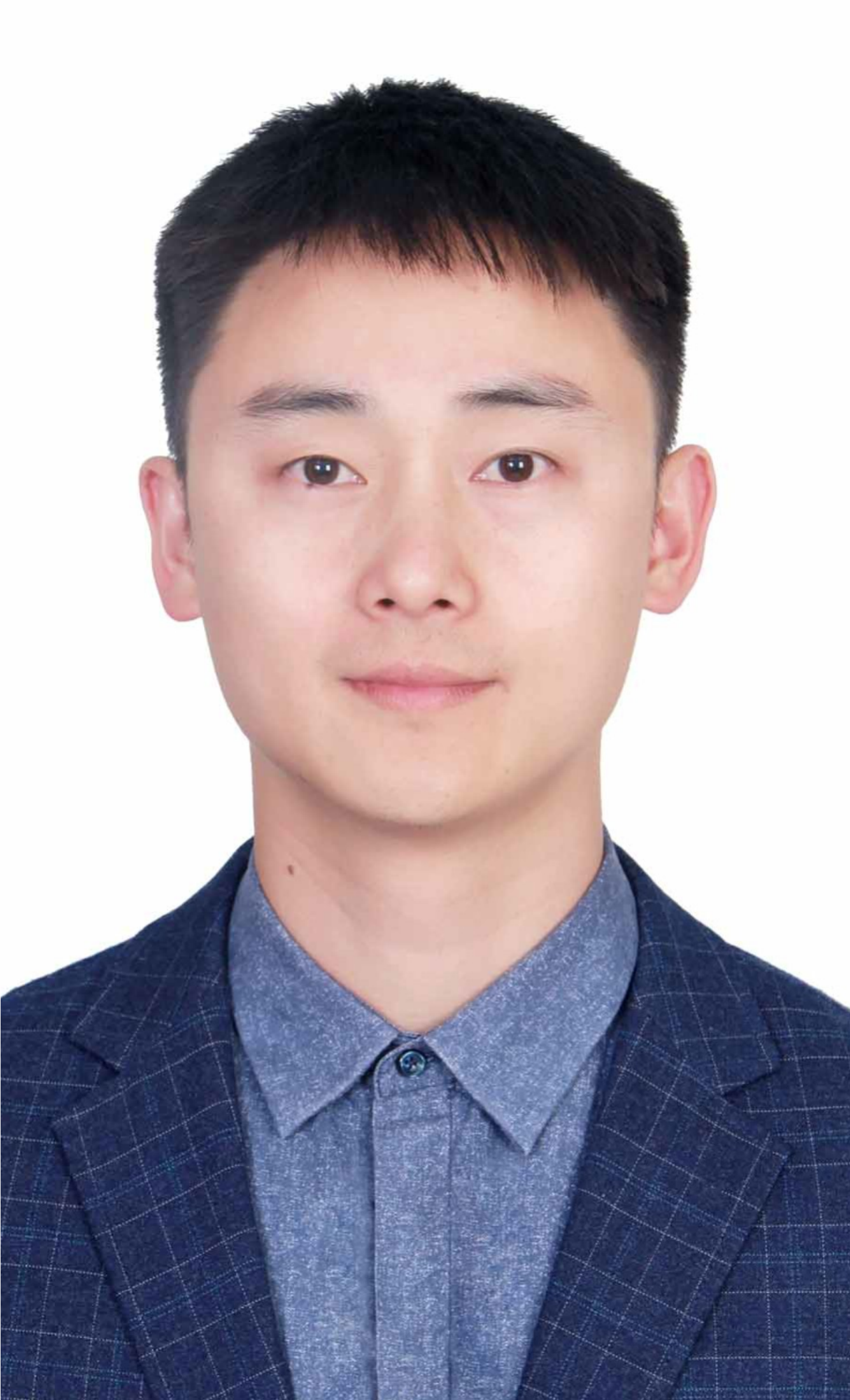}}]{Jinming Xu} received the B.S. degree in Mechanical Engineering from Shandong University,
	China, in 2009 and the Ph.D. degree in Electrical and Electronic Engineering from Nanyang
	Technological University (NTU), Singapore, in
	2016. From 2016 to 2017, he was a Research
	Fellow at the EXQUITUS center, NTU; he was
	also a postdoctoral researcher in the Ira A.
	Fulton Schools of Engineering, Arizona State
	University, from 2017 to 2018, and the School of
	Industrial Engineering, Purdue University, from
	2018 to 2019, respectively. In 2019, he joined Zhejiang University,
	China, where he is currently a Professor with the College of Control
	Science and Engineering. His research interests include distributed
	optimization and control, machine learning and network science. He
	has published over 50 peer-reviewed papers in prestigious journals
	and leading conferences. He has been the Associate Editor of IEEE
	TRANSACTIONS on SIGNAL and INFORMATION PROCESSING OVER NETWORKS.
\end{IEEEbiography}
\begin{IEEEbiography}
	[{\includegraphics[width=1in,height=1.25in,clip,keepaspectratio]{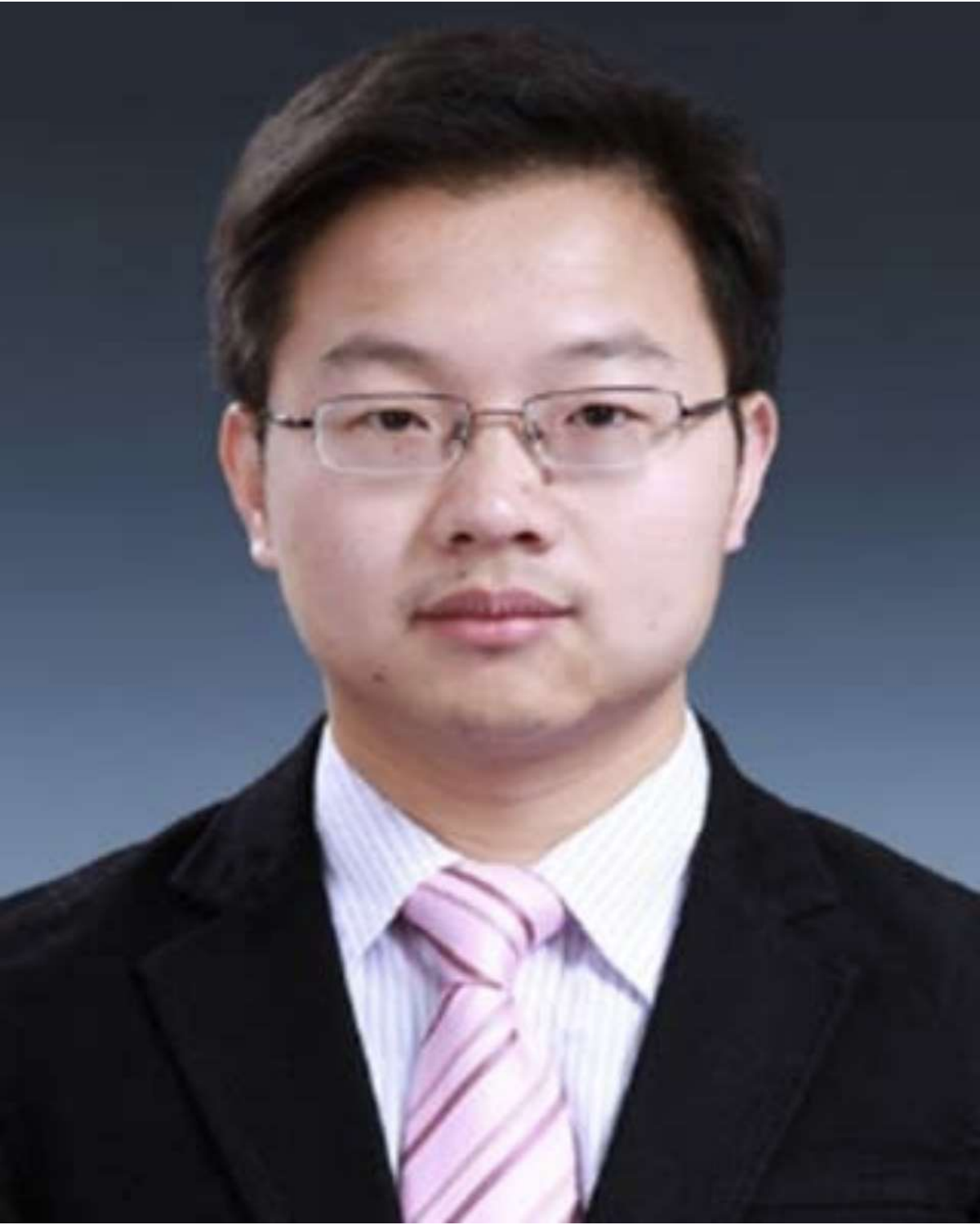}}]{Keyou You}
	(Senior Member, IEEE) received the B.S. degree in statistical science from Sun Yat-sen University, Guangzhou, China, in 2007 and the Ph.D. degree in electrical and electronic engineering from Nanyang Technological University (NTU), Singapore, in 2012. 
	
	After briefly working as a Research Fellow at NTU, he joined Tsinghua University, Beijing, China, where he is currently a Full Professor in the Department of Automation. He held visiting positions with Politecnico di Torino, Turin, Italy, Hong Kong University of Science and Technology, Hong Kong, China, University of Melbourne, Melbourne, Victoria, Australia, and so on. His research interests include the intersections between control, optimization and learning, as well as their applications in autonomous systems. 
	
	Dr. You received the Guan Zhaozhi Award at the 29th Chinese Control Conference in 2010 and the ACA (Asian Control Association) Temasek Young Educator Award in 2019. He received the National Science Funds for Excellent Young Scholars in 2017 and for Distinguished Young Scholars in 2023. He is currently an Associate Editor for \textit{Automatica} and IEEE TRANSACTIONS ON CONTROL OF NETWORK SYSTEMS.
\end{IEEEbiography}
\begin{IEEEbiography}
	[{\includegraphics[width=1in,height=1.25in,clip,keepaspectratio]{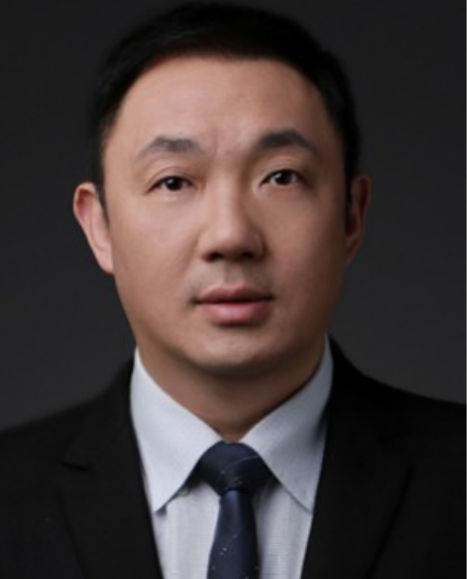}}]{Yang Shi} (Fellow, IEEE) 
	received the B.Sc. and Ph.D. degrees in mechanical engineering and
	automatic control from Northwestern Polytechnical University, Xi’an, China, in 1994 and 1998, respectively, and the Ph.D. degree in electrical and computer engineering from the University of Alberta, Edmonton, AB, Canada, in 2005. He was a Research Associate with the Department of Automation, Tsinghua University, China, from 1998 to 2000. From 2005 to 2009, he was an Assistant Professor and an Associate Professor with the Department of Mechanical Engineering, University of Saskatchewan, Saskatoon, SK, Canada. In 2009, he joined the University of Victoria, and currentlu he is a Professor with the Department of Mechanical Engineering, University of Victoria, Victoria, BC, Canada. His current research interests include networked and distributed systems, model predictive control (MPC), cyber-physical systems (CPS), robotics and mechatronics, navigation and control of autonomous systems (AUV and UAV), and energy system applications.
	
	Dr. Shi is the IFAC Council Member. He is a fellow of ASME, CSME,
	Engineering Institute of Canada (EIC), Canadian Academy of Engineering (CAE), Royal Society of Canada (RSC), and a registered Professional Engineer in British Columbia and Canada. He received the University of Saskatchewan Student Union Teaching Excellence Award in 2007, the Faculty of Engineering Teaching Excellence Award in 2012 at the University of Victoria (UVic), and the 2023 REACH Award for Excellence in Graduate Student Supervision and Mentorship. On research, he was a recipient of the JSPS Invitation Fellowship (short-term) in 2013, the UVic Craigdarroch Silver Medal for Excellence in Research in 2015, the Humboldt Research Fellowship for Experienced Researchers in 2018, CSME Mechatronics Medal in 2023, the IEEE Dr.-Ing. Eugene Mittelmann Achievement Award in 2023,
	the 2024 IEEE Canada Outstanding Engineer Award. He was a Vice-President on Conference Activities of IEEE IES from 2022 to 2025 and the Chair of IEEE IES Technical Committee on Industrial Cyber-Physical Systems. Currently, he is the Editor-in-Chief of IEEE TRANSACTIONS ON INDUSTRIAL ELECTRONICS. He also serves as Associate Editor for \textit{Automatica}, IEEE TRANSACTIONS ON AUTOMATIC CONTROL, and \textit{Annual Review in Controls}.
\end{IEEEbiography}

\end{document}